\documentclass[11pt]{article}

\usepackage[a4paper]{geometry}
\usepackage{hyperref}
\usepackage{graphicx}

\usepackage{amsmath,amsfonts}
\usepackage{bm}
\usepackage{cases}
\usepackage{amsthm}
\usepackage{algorithm}
\usepackage{algorithmicx}
\usepackage{algpseudocode}
\usepackage{soul}
\usepackage{enumerate}
\usepackage{physics}

\usepackage{amssymb}
\usepackage{url}
\usepackage{type1cm}
\usepackage{here}
\usepackage{blkarray}

\usepackage[numbers]{natbib}

\newcommand{\inverse}{{-1}}
\newcommand{\subjectTo}{\text {s.t. }}
\newcommand{\SolSet}{\mathcal{S}^*}
\newcommand{\Real}{\mathbb{R}}

\newcommand{\Curve}{\mathcal{C}}
\newcommand{\Neighborhood}{\mathcal{N}}
\providecommand{\norm}[1]{\left\|#1\right\|}

\newtheorem{assumption}{Assumption}[section]
\newtheorem{theorem}{Theorem}[section]
\newtheorem{lemma}{Lemma}[section]
\newtheorem{proposition}{Proposition}[section]

\usepackage[normalem]{ulem}
\usepackage{easy-todo}

\usepackage[mathlines]{lineno}

\title{An inexact infeasible arc-search interior-point method for linear optimization problems}
\date{2026/01/07}
\author{
  Einosuke Iida\thanks{Department of Mathematical and Computing Science, Institute of Science Tokyo} and
  Makoto Yamashita\thanks{
    Department of Mathematical and Computing Science, Institute of Science Tokyo.
  }
}

\begin{document}

\maketitle

\begin{abstract}
  We propose an inexact infeasible arc-search interior-point method for solving linear optimization problems.
  The method combines an arc-search strategy with inexact solutions to Newton systems and admits a polynomial iteration complexity bound.

  In existing inexact infeasible interior-point methods,
  both the linearization error of the central path and the inexactness of the Newton system accumulate along the search direction,
  which forces the algorithm to take very small steps.
  The proposed method mitigates this effect by using an arc-search strategy:
  the curved search path provides a more accurate approximation of the central path,
  so the step size can remain larger even when the Newton system is solved inexactly.
  As a result,
  the proposed method achieves a provably tighter worst-case iteration bound than existing inexact infeasible line-search methods.

  Numerical experiments on NETLIB benchmark problems demonstrate that the proposed method reduces both the number of iterations and the computation time.
\end{abstract}


\section{Introduction}
Linear optimization problems (LOPs) have had an important role in both theoretical analysis and practical applications, and
many methods have been studied for solving LOPs efficiently.
The first interior-point method (IPM) was proposed by Dikin~\cite{dikin1967iterative},
followed by the first polynomial-time IPM introduced by Karmarkar~\cite{karmarkar1984new}.
The use of IPMs has since been extended to other optimization problems,
for example, second-order cone optimization and semidefinite optimization~\cite{nesterov1994interior}.

Various types of IPMs have been studied.
The IPM with a feasible initial point, called feasible IPM,
was proved to be a polynomial algorithm~\cite{gonzaga1990polynomial,kojima1989primal,mizuno1993adaptive,renegar1988polynomial}
and its best iteration complexity is $\order{\sqrt{n} L}$,
where $n$ is the number of variables and $L$ is the binary length of the input data.
The IPM with the infeasible initial point, called infeasible IPM,
was also proved to have global convergence~\cite{kojima1993primal} and polynomial time complexity~\cite{zhang1994convergence}.
Predictor-corrector algorithms were studied in \cite{Mehrotra1992,lustig1992implementing},
and a two-dimensional search IPM was recently proposed by \citet{vitor2022projected}.

An inexact IPM is one such variant.
This inexactly solves a system of linear equations for the search direction at each iteration.
Inexact IPMs generally perform better on large-scale problems than their exact counterparts,
since they reduce the computational cost per iteration by approximately solving the Newton systems.
An infeasible inexact IPM (II-IPM) was initially proposed by \citet{bellavia1998inexact} for solving constrained systems of equations
and has since been extended to LOPs~\cite{freund1999convergence,kojima1993primal,mizuno1999global}.

Several studies~\cite{freund1999convergence,mizuno1999global,korzak2000convergence} proved the convergence of II-IPMs.
In particular, an II-IPM of \citet{mizuno1999global} has $\order{n^2 L}$ iteration complexity.
In \cite{al2009convergence, monteiro2003convergence},
preconditioned conjugate gradient methods were used
as the inexact linear system solver.
The inexact IPMs have recently gained much attention due to their relevance to quantum computing.
Quantum linear system algorithms (QLSAs) have the potential to solve systems of linear equations quickly.
Inexact IPM using the QLSA, called QIPM, is proposed in several papers \cite{kerenidis2020quantum,wu2023inexact,mohammadisiahroudi2024efficient,mohammadisiahroudi2025inexact}.

One main difficulty in designing inexact infeasible IPMs is that the step size becomes severely restricted when the Newton direction is computed inexactly.
In existing methods,
the search direction is based on a linearization of the central path,
and the step size must take into account both the linearization error of the central-path equation and the error introduced by solving the Newton system inexactly.
These errors accumulate along a straight search direction,
forcing the step size to be very small and resulting in relatively loose worst-case iteration bounds $\order{n^2 L}$.

On the other hand,
studies to reduce the number of iterations in IPMs have also contributed to improving the numerical performance.
Higher-order IPMs using second-order or higher derivatives have been studied~\cite{monteiro1990polynomial,Mehrotra1992,gondzio1996multiple,lustig1992implementing,espaas2022interior},
but these sometimes have a worse polynomial bound,
or the analysis of computational complexity is not simple.
An arc-search IPM proposed by Yang~\cite{yang2011polynomial} is one of the higher-order IPMs.
(IPMs that find the next iteration point on a straight line are called line-search IPMs in this paper.)
Arc-search IPMs employ an ellipsoidal arc to find the next iteration point.
Since the central path toward an optimal solution is generally a smooth curve,
the ellipsoidal arc can approximate the central path better than the straight line,
and a reduction in the number of iterations can be expected.
\citet{yang2018arc} showed that
their arc-search IPM has the iteration complexity of $\order{n L}$,
the same as that of the line-search IPM~\cite{wright1997primal}.
In addition,
Yang~\cite{yang2018two} improved it to $\order{\sqrt{n} L}$ under the assumption that the variables $x^k$ and $s^k$ are bounded below and away from zeros for all iterations until a stopping criterion is satisfied.
Numerical experiments in \cite{Yang2017,yang2018arc} demonstrated that the number of iterations of an LOP was reduced compared to existing methods.
The arc-search IPMs have been applied to various optimization problems,
including second-order cone~\cite{yang2017arc}, semidefinite~\cite{zhang2019primal}, convex~\cite{yang2023polynomial}, and nonlinear optimization problems~\cite{Yamashita2021}.
In the work of Iida and Yamashita~\cite{iida2024infeasible},
an arc-search IPM with Nesterov's restarting strategy further reduced the number of iterations.

Arc-search IPMs update the iterate along a curve that more accurately approximates the central path.
Because the curved trajectory suppresses the linearization error,
the step size can remain substantially larger even when the Newton system is solved inexactly.
This property enables a sharper iteration complexity guarantee.

In this paper,
we propose a novel inexact infeasible arc-search interior-point method (II-arc).
To the best of the authors' knowledge,
this paper is the first paper that discusses an inexact IPM with the arc-search framework.
We show that the method is a polynomial-time algorithm and establish an improved iteration complexity bound of $\order{n ^{1.5} L}$,
tightening the best-known bound for inexact infeasible IPMs~\cite{mizuno1999global,mohammadisiahroudi2024efficient} by a factor of $n^{0.5}$.
We conducted numerical experiments using the conjugate gradient method~(CG) as an inexact linear equation solver.
The numerical results are consistent with the convergence analysis,
and the proposed method (II-arc) reduced the number of iterations by almost half for many instances.

This paper is organized as follows.
Section~\ref{section_preliminaries} introduces the standard form of LOPs and the formulas necessary for II-arc.
In Section~\ref{section_proposed_method},
we describe the proposed method,
and in Section~\ref{section_theoretical_proof},
we discuss the convergence and the polynomial iteration complexity.
Section~\ref{section_numerical_experiments} provides the results of the numerical experiments and the discussion.
Finally,
Section~\ref{section_conclusion} concludes the paper and discusses future directions.

\subsection{Notations}
We use $x_i$ to denote the $i$-th element of a vector $x$.
The Hadamard product of two vectors $u$ and $v$ is denoted by $u \circ v$.
The vector of all ones and the identity matrix are denoted by $e$ and $I$, respectively.
We use the capital character $X \in \Real^{n \times n}$ as the diagonal matrix whose diagonal elements are taken from the vector $x \in \Real^n$.
For set $B$,
we denote the cardinality of the set by $\abs*{B}$.
Given a matrix $A \in \Real^{m \times n}$ and a set $B \subseteq \{1, \ldots, n\}$,
the matrix $A_B$ is the submatrix consisting of the columns $\left\{ A_i: i \in B \right\}$.
Similarly,
given a vector $v \in \mathbb{R}^n$ and a set $B \subseteq\{1, \ldots, n\}$ where $\abs*{B} = m \le n$,
the matrix $V_B \in \Real^{m \times m}$ is the diagonal submatrix consisting of the elements $\left\{ v_i: i \in B \right\}$.
We use $\norm{x}_2 = (\sum_i x_i^2)^{1/2}$,
$\norm{x}_{\infty} = \max_i \abs*{x_i}$ and $\norm{x}_1 = \sum_i \abs*{x_i}$ for the Euclidean norm,
the infinity norm and the $\ell_1$ norm of a vector $x$,
respectively.
For simplicity,
we denote $\norm{x} = \norm{x}_2$.
For a matrix $A \in \Real^{m \times n}$, $\|A\|$ denotes the operator norm associated with the Euclidean norm;
$\|A\|=$ $\max_{\|z\|=1}\|A z\|$.
\section{Preliminaries}
\label{section_preliminaries}
In this paper, we consider an LOP of the standard form:
\begin{equation}
  \label{problem_main}
  \min_{x \in \Real^n} c^\top x, \quad \subjectTo A x = b, \quad x \ge 0,
\end{equation}
where $A \in \Real^{m \times n}$ with $m \le n$, $b \in \Real^m$, and $c \in \Real^n$
are input data.
The associated dual problem of \eqref{problem_main} is
\begin{equation}
  \label{problem_dual}
  \max_{y \in \Real^m, s \in \Real^n} b^\top y, \quad \subjectTo A^\top y+s=c, \quad s \ge 0,
\end{equation}
where $y$ and $s$ are the dual variable vector and the dual slack vector, respectively.
Let $\SolSet$ be the set of the optimal solutions to \eqref{problem_main} and \eqref{problem_dual}.
When $(x^*, y^*, s^*) \in \SolSet$,
it is well-known that $(x^*, y^*, s^*)$ satisfies
the KKT conditions:
\begin{subequations}
  \label{KKT_conditions}
  \begin{align}
    A x^*            & = b                        \\
    A^\top y^* + s^* & = c                        \\
    (x^*,s^*)        & \ge 0                      \\
    x_i^* s_i^*      & = 0, \quad i = 1,\ldots,n.
  \end{align}
\end{subequations}

We denote the primal and dual residuals in \eqref{problem_main} and \eqref{problem_dual} as
\begin{subequations}
  \label{residuals_constraints}
  \begin{align}
    r_b(x)    & = A x - b \label{residual_main}           \\
    r_c(y, s) & = A^\top y + s - c, \label{residual_dual}
  \end{align}
\end{subequations}
and define the duality measure as
\begin{equation}
  \mu = \frac{x^\top s}{n}.
  \label{def_mu}
\end{equation}

Letting $\zeta \ge 0$,
we define the set of $\zeta$-optimal solutions as
\begin{equation}
  \SolSet_\zeta = \left\{
  \left(x, y, s\right)\in \Real^{2n + m} \mid
  (x,s) \ge 0, \,
  \mu \le \zeta, \,
  \norm{(r_b(x), r_c(y,s))} \le \zeta
  \right\}.
  \label{def_zeta_optimal_SolSet}
\end{equation}
From the KKT conditions~\eqref{KKT_conditions},
we know $\SolSet \subset \SolSet_\zeta$.

In this paper,
we make the following assumptions for the primal-dual pair~\eqref{problem_main} and \eqref{problem_dual}.
These assumptions are common ones in the context of IPMs
and are used in many papers
(for example, see \cite{wright1997primal,yang2020arc}).
\begin{assumption}
  \label{assumption_opt_sol_bounded}
  There exists a constant $\omega \geq 1$ such that
  \begin{equation}
    \label{def_omega}
    \omega \ge \max \left\{\norm{(x^*, s^*)}_\infty \mid (x^*, y^*, s^*) \in \SolSet \right\}.
  \end{equation}
\end{assumption}
\begin{assumption}
  \label{assumption_full_row_rank}
  $A$ is a full-row rank matrix, i.e., $\rank(A) = m$.
\end{assumption}

IPMs are iterative methods,
so we denote the $k$th iteration by $(x^k, y^k, s^k)\in \Real^n \times \Real^m \times \Real^n$ and
the initial point by $(x^0, y^0, s^0)$.
We denote the duality measure of $k$th iteration as $\mu_k = (x^k)^\top s^k / n$.

Similarly to Wright~\cite[Chapter~6, p112]{wright1997primal},
we choose the initial point as follows to prove the polynomial convergence:
\begin{equation}
  \label{def_initial_point}
  (x^0, y^0, s^0) = \omega (e, 0, e),
\end{equation}
where $\omega$ is a scalar satisfying \eqref{def_omega}.

Given a strictly positive iteration $(x^k, y^k, s^k)$ such that $(x^k, s^k) > 0$,
the strategy of an infeasible IPM is to trace a smooth curve called an approximate central path:
\begin{equation}
  \label{def_ellipsoid}
  \Curve^k = \left\{(x(t), y(t), s(t)) \mid t \in (0,1], (x(1), y(1), s(1)) = (x^k, y^k, s^k) \right\},
\end{equation}
where $(x(t), y(t), s(t))$ is the unique solution of the following system
\begin{subequations}
  \label{curve_to_optimal_solution}
  \begin{align}
     & A x(t) - b = t \ r_b(x^k),                  \\
     & A^\top y(t) + s(t) - c = t \ r_c(y^k, s^k), \\
     & x(t) \circ s(t) = t (x^k \circ s^k),        \\
     & (x(t), s(t))>0.
  \end{align}
\end{subequations}
As $t \rightarrow 0$,
$(x(t), y(t), s(t))$ converges to an optimal solution $(x^*, y^*, s^*) \in \SolSet$.

Since \eqref{curve_to_optimal_solution} does not admit an analytical solution,
arc-search IPMs~\cite{yang2020arc} use an ellipse $\mathcal{E}^k$ to approximate the curve $\Curve^k$,
where
$$
  \mathcal{E}^k = \left\{
  (x(\alpha), y(\alpha), s(\alpha)):
  \begin{aligned}
     & (x(\alpha), y(\alpha), s(\alpha))=\overrightarrow{a} \cos (\alpha)+\overrightarrow{b} \sin (\alpha)+\overrightarrow{c}, \\
     & (x(0), y(0), s(0)) = (x^k, y^k, s^k)                                                                                    \\
  \end{aligned}
  \right\},
$$
$\overrightarrow{a}, \overrightarrow{b} \in \Real^{2n+m}$ are the axes of the ellipse
and $\overrightarrow{c} \in \Real^{2n+m}$ is the center of it.
The values of $(x(\alpha), y(\alpha), s(\alpha)) \in \mathcal{E}^k$
for an angle $\alpha \in [0, \pi / 2]$ is obtained by the following relation~\cite[Theorem~5.1]{yang2020arc}:
\begin{subequations}
  \begin{align*}
    x(\alpha) & = \ x^k - \dot{x}\sin(\alpha)+\ddot{x}(1-\cos(\alpha)), \\
    y(\alpha) & = \ y^k - \dot{y}\sin(\alpha)+\ddot{y}(1-\cos(\alpha)), \\
    s(\alpha) & = \ s^k - \dot{s}\sin(\alpha)+\ddot{s}(1-\cos(\alpha)).
  \end{align*}
\end{subequations}

In arc-search frameworks,
we update the iteration points as $(x^{k+1}, y^{k+1}, s^{k+1}) = (x(\alpha), y(\alpha), s(\alpha))$.
Here,
$(\dot{x}, \dot{y}, \dot{s})$ and $(\ddot{x}, \ddot{y}, \ddot{s})$ are
the first and second derivatives

of \eqref{curve_to_optimal_solution} with respect to $t$, respectively.
Therefore,
they can be obtained as the solutions of the following.
\begin{align}
  \label{first_derivative_original}
  \begin{bmatrix}
    A   & 0      & 0   \\
    0   & A^\top & I   \\
    S^k & 0      & X^k
  \end{bmatrix} \left[\begin{array}{l}
                          \dot{x} \\
                          \dot{y} \\
                          \dot{s}
                        \end{array}\right]
   & =
  \left[\begin{array}{c}
            r_b(x^k)      \\
            r_c(y^k, s^k) \\
            x^k \circ s^k
          \end{array}\right], \\
  \label{second_derivative_original}
  \begin{bmatrix}
    A   & 0      & 0   \\
    0   & A^\top & I   \\
    S^k & 0      & X^k
  \end{bmatrix} \left[\begin{array}{l}
                          \ddot{x} \\
                          \ddot{y} \\
                          \ddot{s}
                        \end{array}\right]
   & =
  \left[\begin{array}{c}
            0 \\
            0 \\
            - 2 \dot{x} \circ \dot{s}
          \end{array}\right].
\end{align}

Lastly,
we define a neighborhood of the approximate central
path~\cite[Chapter~6]{wright1997primal}:
\begin{equation}
  \Neighborhood(\gamma_1, \gamma_2) := \left\{
  (x, y, s) \mid
  \begin{aligned}
     & (x, s)>0, \, x_i s_i \ge \gamma_1 \mu \text{ for } i \in\{1, \ldots, n\},              \\
     & \norm{(r_b(x), r_c(y, s))} \le [\norm{(r_b(x^0), r_c(y^0, s^0))} / \mu_0] \gamma_2 \mu \\
  \end{aligned}
  \right\},
  \label{def_neighborhood}
\end{equation}
where $\gamma_1 \in (0, 1)$ and $\gamma_2 \ge 1$ are given parameters,
and $\norm{(r_b(x), r_c(y, s))}$ is the norm of the vertical
concatenation of $r_b(x)$ and $r_c(y, s)$.

Note that the initial point $(x^0, y^0, s^0)$ defined by \eqref{def_initial_point} always satisfies $(x^0, y^0, s^0) \in \Neighborhood(\gamma_1, \gamma_2)$.
This neighborhood will be used in the convergence analysis.
\section{The proposed method}
\label{section_proposed_method}
In this section,
we propose an II-arc method.

To ensure that the iterates $(x, s)$ remain strictly within the positive orthant,
we introduce a perturbation to \eqref{first_derivative_original},
following the approach of \citet{kojima1993primal},
as described below:
\begin{equation}
  \label{first_derivative_perturbed}
  \begin{bmatrix}
    A   & 0      & 0   \\
    0   & A^\top & I   \\
    S^k & 0      & X^k
  \end{bmatrix} \left[\begin{array}{l}
      \dot{x} \\
      \dot{y} \\
      \dot{s}
    \end{array}\right]
  = \left[\begin{array}{c}
      r_b(x^k)      \\
      r_c(y^k, s^k) \\
      x^k \circ s^k - \sigma \mu_k e
    \end{array}\right],
\end{equation}
where $\sigma \in (0, 1]$ is the constant called centering parameter.
In the subsequent discussion,
$(\dot{x}, \dot{y}, \dot{s})$ denotes the solution of \eqref{first_derivative_perturbed}.
The proposed method solves \eqref{first_derivative_perturbed} and \eqref{second_derivative_original} inexactly
at each iteration to obtain the ellipsoidal approximation.

Several approaches can be considered for solving the Newton system~\eqref{first_derivative_perturbed},
such as the augmented system and the Newton equation system (also known as the normal equation system, NES)~\cite{bellavia2004convergence}.
The NES formula of \eqref{first_derivative_perturbed} is
\begin{equation}
  M^k \dot{y} = \rho_1^k,
  \label{first_derivative_NES}
\end{equation}
where
\begin{subequations}
  \begin{align}
    M^k      & = A (D^k)^2 A^\top,                                                                             \\
    D^k      & = (X^k)^{\frac{1}{2}} (S^k)^{-\frac{1}{2}},                                                     \\
    \rho_1^k & = A (D^k)^2 r_c(y^k, s^k) + r_b(x^k) - A (S^k)^\inverse (x^k \circ s^k - \sigma \mu_k e) \notag \\
             & = A (D^k)^2 A^\top y^k - A (D^k)^2 c + \sigma \mu_k A (S^k)^\inverse e + A x^k - b.
    \label{def_NES_rho_1}
  \end{align}
  \label{def_NES_constants}
\end{subequations}
When we solve \eqref{first_derivative_NES} exactly and obtain $\dot{y}$,
we can compute the other components $\dot{x}$ and $\dot{s}$ of the solution in \eqref{first_derivative_perturbed}.

Let $\tilde{\dot{y}}$ be an inexact solution of $\dot{y}$.
We denote the error of $\tilde{\dot{y}}$ by $r_1^k := M^k \tilde{\dot{y}} - \rho_1^k = M^k \left(\tilde{\dot{y}} - \dot{y}\right)$.
Then, the inexact solution $\tilde{\dot{y}}$ satisfies
\begin{equation}
  M^k \tilde{\dot{y}} = \rho_1^k + r_1^k,
  \label{inexact_first_derivative_NES}
\end{equation}
We can compute the other inexact components $\tilde{\dot{x}}$ and $\tilde{\dot{s}}$ from $\tilde{\dot{y}}$,
and these inexact components satisfy the following conditions~\cite{mohammadisiahroudi2024efficient}:
\begin{subequations}
  \begin{align}
    A \tilde{\dot{x}}                         & = r_b(x^k) + r_1^k, \label{inexact_first_derivative_NES_main_residual} \\
    A^\top \tilde{\dot{y}} + \tilde{\dot{s}}  & = r_c(y^k, s^k),                                                       \\
    S^k \tilde{\dot{x}} + X^k \tilde{\dot{s}} & = X^k s^k - \sigma \mu_k e.
    \label{inexact_first_derivative_NES_duality}
  \end{align}
\end{subequations}

\citet{bellavia2004convergence} pointed out that in the NES framework,
the residual must decrease at least $\order{\sqrt{n} \log n}$ times faster than the duality measure $\mu_k$ in order to guarantee convergence.
\citet{oliveira2005new} proposed a preconditioned matrix for NES,
and \citet{monteiro2004uniform} prove that the condition number of the preconditioned matrix can be bounded above by $\order{1}$ that is independent of the diagonal scaling $D_k$ and hence independent of the duality gap $\mu_k$,
while the one of NES is bounded above by $\order{1 / \mu^2}$~\cite{mohammadisiahroudi2018improvements}.
The modified NES formulation using this preconditioner is called MNES~\cite{mohammadisiahroudi2024efficient}.
The preconditioner relaxes the required decreasing of the residual,
and facilitates the derivation of iteration complexity bounds under less stringent assumptions on residual accuracy.
For this reason, this paper adopts the MNES formulation in its complexity analysis.

Since $A$ is full row rank from Assumption~\ref{assumption_full_row_rank},
we can choose an arbitrary basis $\hat{B} \subset \{1,2,\dots,n\}$ such that $\abs*{\hat{B}} = m$ and $A_{\hat{B}} \in \Real^{m \times m}$ is nonsingular.
Such a $\hat{B}$ can be obtained with, for example, the maximum weight basis algorithm~\cite{monteiro2003convergence},
which requires $\order{n^2 m}$ computational costs.
Now we can adapt \eqref{first_derivative_NES} to
\begin{equation}
  \hat{M}^k \dot{z} = \hat{\rho}_1^k,
  \label{first_derivative_MNES}
\end{equation}
where
\begin{subequations}
  \label{def_MNES_constants}
  \begin{align}
    \hat{M}^k      & = (D_{\hat{B}}^k)^\inverse A_{\hat{B}}^\inverse M^k ((D_{\hat{B}}^k)^\inverse A_{\hat{B}}^\inverse)^\top, \label{def_MNES_coef_matrix} \\
    D_{\hat{B}}^k  & = (X^k_{\hat{B}})^\frac{1}{2} (S_{\hat{B}}^k)^{-\frac{1}{2}},                                                                          \\
    \hat{\rho}_1^k & = (D_{\hat{B}}^k)^\inverse A_{\hat{B}}^\inverse \rho_1^k.
  \end{align}
\end{subequations}
The inexact solution $\tilde{\dot{z}}$ of \eqref{first_derivative_MNES} satisfies
\begin{equation}
  \hat{M}^k \tilde{\dot{z}} = \hat{\rho}_1^k + \hat{r}^k_1,
  \label{inexact_first_derivative_MNES}
\end{equation}
where $\hat{r}_1^k$ is the error of $\tilde{\dot{z}}$ defined as
\begin{equation*}
  \hat{r}_1^k := \hat{M}^k \tilde{\dot{z}} - \hat{\rho}_1^k = \hat{M}_k \left(\tilde{\dot{z}} - \dot{z}\right).
\end{equation*}
We discuss the range that the error $\hat{r}_1^k$ should satisfy for the convergence analysis below.

Then,
we can obtain the first derivative $(\tilde{\dot{x}}, \tilde{\dot{y}}, \tilde{\dot{s}})$
from the inexact solution in \eqref{inexact_first_derivative_MNES}
and the steps below:
\begin{subequations}
  \label{resolution_first_derivative_from_MNES}
  \begin{align}
    \tilde{\dot{y}} & = \left(\left(D_{\hat{B}}^k\right)^{-1} A_{\hat{B}}^{-1}\right)^\top \tilde{\dot{z}}  \\
    \tilde{\dot{s}} & = r_c(y^k, s^k) - A^T \tilde{\dot{y}}                                                 \\
    v_1^k           & = \left(v_{\hat{B}}^k, v_{\hat{N}}^k\right)=\left(D_{\hat{B}}^k \hat{r}^k_1, 0\right) \\
    \tilde{\dot{x}} & = x^k - (D^k)^2 \tilde{\dot{s}} - \sigma \mu_k (S^k)^\inverse e - v_1^k.
  \end{align}
\end{subequations}

We also apply the MNES formulation to the second derivative~\eqref{second_derivative_original}.
Letting
\begin{align*}
  \rho_2^k  = 2 A (S^k)^\inverse \tilde{\dot{x}} \circ \tilde{\dot{s}}, \qquad
  \hat{\rho}_2^k  = (D_{\hat{B}}^k)^\inverse A_{\hat{B}}^\inverse \rho_2^k,
\end{align*}
we have
\begin{equation}
  \hat{M}^k \ddot{z} = \hat{\rho}_2^k
  \label{second_derivative_MNES}
\end{equation}
with the same definition of $\hat{M}^k$ as in \eqref{def_MNES_coef_matrix}.
We use $\tilde{\ddot{z}}$ to denote the inexact solution of \eqref{second_derivative_MNES},
then we have
\begin{equation}
  \hat{M}^k \tilde{\ddot{z}} = \hat{\rho}_2^k + \hat{r}_2^k,
  \label{inexact_second_derivative_MNES}
\end{equation}
where $\hat{r}_2^k$ is defined as $\hat{r}_2^k := \hat{M}_k \left(\tilde{\ddot{z}} - \ddot{z}\right)$.
Similarly to \eqref{resolution_first_derivative_from_MNES},
we compute
the inexact second derivative $(\tilde{\ddot{x}}, \tilde{\ddot{y}}, \tilde{\ddot{s}})$
as follows:
\begin{subequations}
  \begin{align*}
    \tilde{\ddot{y}} & = \left(\left(D_{\hat{B}}^k\right)^{-1} A_{\hat{B}}^{-1}\right)^\top \tilde{\ddot{z}},         \\
    \tilde{\ddot{s}} & = - A^T \tilde{\ddot{y}},                                                                      \\
    v_2^k            & = \left(v_{\hat{B}}^k, v_{\hat{N}}^k\right)=\left(D_{\hat{B}}^k \hat{r}_2^k, 0\right),         \\
    \tilde{\ddot{x}} & = - (D^k)^2 \tilde{\ddot{s}} - 2 (S^k)^\inverse \tilde{\dot{x}} \circ \tilde{\dot{s}} - v_2^k.
  \end{align*}
  \label{resolution_second_derivative_from_MNES}
\end{subequations}

Using the inexact solutions obtained above,
the next iteration point on the ellipsoidal arc $\mathcal{E}^k$ is computed with the following formula:
\begin{subequations}
  \label{def_variable_alpha_with_inexact_derivatives}
  \begin{align}
    x^k(\alpha) = & \ x^k-\tilde{\dot{x}}\sin(\alpha)+\tilde{\ddot{x}}(1-\cos(\alpha)), \\
    y^k(\alpha) = & \ y^k-\tilde{\dot{y}}\sin(\alpha)+\tilde{\ddot{y}}(1-\cos(\alpha)), \\
    s^k(\alpha) = & \ s^k-\tilde{\dot{s}}\sin(\alpha)+\tilde{\ddot{s}}(1-\cos(\alpha)).
  \end{align}
\end{subequations}

To present the framework of the proposed method,
we prepare the following functions:
\begin{subequations}
  \label{def_G_g_h}
  \begin{align}
    G_i^k(\alpha) & = x_i^k(\alpha) s_i^k(\alpha) - \gamma_1 \mu_k(\alpha) \text { for } i \in\{1, \ldots, n\}, \\
    g^k(\alpha)   & = x^k(\alpha)^\top s^k(\alpha) - (1 - \sin(\alpha))(x^k)^\top s^k,                          \\
    h^k(\alpha)   & = \left(1 - (1-\beta) \sin(\alpha)\right)(x^k)^\top s^k-x^k(\alpha)^\top s^k(\alpha).
  \end{align}
\end{subequations}
Here, $h^k(\alpha) \ge 0$ corresponds to the Armijo condition with respect to the duality gap $\mu$.
In Section~\ref{section_theoretical_proof},
we will show that the sequence $\{(x^k, y^k, s^k)\}$ converges to an optimal solution by selecting a step size $\alpha$ satisfying:
\begin{equation}
  \label{conditions_G_g_h_no_less_than_0}
  G_i^k(\alpha) \ge 0 \text { for } i \in\{1, \ldots, n\}, \quad
  g^k(\alpha) \ge 0, \quad
  h^k(\alpha) \ge 0.
\end{equation}

Under this condition,
the next lemma ensures that a next iteration point $(x^k(\alpha), y^k(\alpha), s^k(\alpha))$ is in the neighborhood $\Neighborhood(\gamma_1, \gamma_2)$.
\begin{lemma}
  \label{lemma_in_neighborhood}
  Assume a step length $\alpha \in (0, \pi / 2]$ satisfies $G_i^k(\alpha) \ge 0$ and $g^k(\alpha) \ge 0$.
  Then,
  $(x^k(\alpha), y^k(\alpha), s^k(\alpha)) \in \Neighborhood(\gamma_1, \gamma_2)$.
\end{lemma}
This lemma can be proved using the same approach as \citet[Lemma~4.5]{mohammadisiahroudi2024efficient}.

Lastly,
we discuss the range in which the sequence with inexact solutions
attain the polynomial iteration complexity.
We assume the following inequality for the error $\hat{r}_1^k$ of \eqref{inexact_first_derivative_MNES} and $\hat{r}_2^k$ of \eqref{inexact_second_derivative_MNES}:
\begin{equation}
  \norm{\hat{r}_i^k} \leq \eta \frac{\sqrt{\mu_k}}{\sqrt{n}},
  \quad \forall i \in \{1, 2\}
  \label{def_upper_derivatives_residual_MNES}
\end{equation}
where $\eta \in [0, 1)$ is an enforcing parameter.

To prove the polynomial iteration complexity of the proposed algorithm in Proposition~\ref{proposition_lower_bound_of_step_size} below,
we set the parameters so that
\begin{subequations}
  \label{parameter_conditions}
  \begin{align}
    (1 - \gamma_1) \sigma - (1 + \gamma_1) \eta & > 0, \label{parameter_condition_for_G_i}                                   \\
    \beta                                       & > \sigma + \eta. \label{parameter_condition_beta_more_than_sigma_plus_eta}
  \end{align}
\end{subequations}

Algorithm~\ref{algorithm_II_arc_IPM} gives the framework of the proposed method (II-arc).
\begin{algorithm}[ht]
  \caption{The inexact infeasible arc-search interior-point method (II-arc)}
  \label{algorithm_II_arc_IPM}
  \begin{algorithmic}[1]
    \Require $\zeta > 0$,
    $\gamma_1 \in (0,1)$,
    $\gamma_2 \ge 1$,
    $\sigma, \eta, \beta$ satisfying \eqref{parameter_conditions} and
    an initial point $(x^0, y^0, s^0)$ meeting \eqref{def_initial_point}.
    \Ensure $\zeta$-optimal solution $(x^k, y^k, s^k)$
    \State $k \leftarrow 0$
    \While {$(x^k, y^k, s^k) \notin S_\zeta$}  \label{line_algo_II_arc_search_checking_stop}
    \State $\mu_k \leftarrow (x^k)^\top s^k / n$
    \State Calculate $(\tilde{\dot{x}}, \tilde{\dot{y}}, \tilde{\dot{s}})$ by solving \eqref{first_derivative_MNES} inexactly satisfying \eqref{def_upper_derivatives_residual_MNES}.
    \State Calculate $(\tilde{\ddot{x}}, \tilde{\ddot{y}}, \tilde{\ddot{s}})$ by solving \eqref{second_derivative_MNES} inexactly satisfying \eqref{def_upper_derivatives_residual_MNES}.
    \label{line_algo_II_arc_search_calculate_second_derivative}
    \State $\alpha_k \leftarrow \max \left\{\alpha \in (0, \pi / 2] \mid \alpha \text{ satisfies } \eqref{conditions_G_g_h_no_less_than_0}\right\}$
    \label{line_algo_II_arc_search_decide_step_size}
    \State Set $(x^{k+1}, y^{k+1}, s^{k+1}) = (x^k(\alpha_k), y^k(\alpha_k), s^k(\alpha_k))$ by \eqref{def_variable_alpha_with_inexact_derivatives}.
    \State $k \leftarrow k+1$
    \EndWhile
  \end{algorithmic}
\end{algorithm}
\section{Theoretical proof}
\label{section_theoretical_proof}
In this section,
we prove the convergence of Algorithm~\ref{algorithm_II_arc_IPM} and its polynomial iteration complexity.
Our analysis is close to Mohammadisiahroudi et al.~\cite{mohammadisiahroudi2024efficient},
but it also employs properties of arc-search IPMs.

First,
we evaluate the constraint residuals \eqref{residuals_constraints}.
From \eqref{inexact_first_derivative_MNES} and \eqref{resolution_first_derivative_from_MNES},
the residual appears only in the last equation as a term
$S^k v_1^k$,
as the following lemma shows.
\begin{lemma}
  \label{lemma_inexact_solution_MNES_conditions}
  For the inexact first derivative $(\tilde{\dot{x}}, \tilde{\dot{y}}, \tilde{\dot{s}})$ of \eqref{curve_to_optimal_solution}
  obtained by the inexact solution of \eqref{first_derivative_MNES} and the steps in \eqref{resolution_first_derivative_from_MNES},
  we have
  \begin{subequations}
    \begin{align}
      A \tilde{\dot{x}}                         & = r_b(x^k), \label{inexact_first_derivative_MNES_main_residual}      \\
      A^\top \tilde{\dot{y}} + \tilde{\dot{s}}  & = r_c(y^k, s^k), \label{inexact_first_derivative_MNES_dual_residual} \\
      S^k \tilde{\dot{x}} + X^k \tilde{\dot{s}} & = X^k s^k - \sigma \mu_k e - S^k v_1^k.
      \label{inexact_first_derivative_MNES_duality}
    \end{align}
  \end{subequations}
\end{lemma}
Lemma~\ref{lemma_inexact_solution_MNES_conditions} can be proved
from \eqref{first_derivative_MNES} and \eqref{resolution_first_derivative_from_MNES}
in the same way as Mohammadisiahroudi et al.~\cite[Lemma~4.1]{mohammadisiahroudi2024efficient},
thus we omit the proof.
As in Lemma~\ref{lemma_inexact_solution_MNES_conditions},
$(\tilde{\ddot{x}}, \tilde{\ddot{y}}, \tilde{\ddot{s}})$ obtained by \eqref{inexact_second_derivative_MNES} and \eqref{resolution_second_derivative_from_MNES} satisfies
\begin{subequations}
  \label{indexact_second_derivative_conditions}
  \begin{align}
    A \tilde{\ddot{x}}                          & = 0, \label{inexact_second_derivative_main_residual}    \\
    A^\top \tilde{\ddot{y}} + \tilde{\ddot{s}}  & = 0, \label{inexact_second_derivative_dual_residual}    \\
    S^k \tilde{\ddot{x}} + X^k \tilde{\ddot{s}} & = -2 \tilde{\dot{x}} \circ \tilde{\dot{s}} - S^k v_2^k.
    \label{inexact_second_derivative_duality}
  \end{align}
\end{subequations}
Therefore,
the following lemma holds from \eqref{inexact_first_derivative_MNES_main_residual},
\eqref{inexact_first_derivative_MNES_dual_residual},
\eqref{inexact_second_derivative_main_residual} and \eqref{inexact_second_derivative_dual_residual}
due to \eqref{def_variable_alpha_with_inexact_derivatives}.
\begin{lemma}[{\cite[Lemma~7.2]{yang2020arc}}]
  \label{lemma_decrease_constraint_residuals}
  For each iteration $k$,
  the following relations hold.
  \begin{align*}
    r_b(x^{k+1})          & = r_b(x^k)\left(1 - \sin(\alpha_k)\right),       \\
    r_c(y^{k+1}, s^{k+1}) & = r_c(y^k, s^k) \left(1 - \sin(\alpha_k)\right).
  \end{align*}
\end{lemma}

For the following discussions,
we introduce the notation:
$$
  \nu_k = \prod_{i=0}^{k-1} (1 - \sin(\alpha_i)).
$$
From Lemma~\ref{lemma_decrease_constraint_residuals},
we can obtain
\begin{subequations}
  \label{residuals_decreasing}
  \begin{align}
    r_b(x^k)      & = \nu_k r_b(x^0)      \\
    r_c(y^k, s^k) & = \nu_k r_c(y^0, s^0)
  \end{align}
\end{subequations}

In the next proposition,
we prove the existence of the lower bound of the step size $\alpha_k$
to guarantee that Algorithm~\ref{algorithm_II_arc_IPM} is well defined.
\begin{proposition}
  \label{proposition_lower_bound_of_step_size}
  Let $\{(x^k, y^k, s^k)\}$ be the sequence generated by Algorithm~\ref{algorithm_II_arc_IPM}.
  Then,
  there exists $\hat{\alpha} > 0$ satisfying
  \eqref{conditions_G_g_h_no_less_than_0} for any $\alpha_k \in (0, \hat{\alpha}]$ and
  $$
    \sin(\hat{\alpha}) = \frac{C}{n^{1.5}},
  $$
  where $C$ is a positive constant.
\end{proposition}
The proof of Proposition~\ref{proposition_lower_bound_of_step_size} will be given later.
For this proof,
we first evaluate $x^k$ and $s^k$ with the $\ell_1$ norm.
\begin{lemma}
  \label{lemma_upper_nu_x_s}
  There is a positive constant $C_1$ such that
  \begin{equation}
    \nu_k \norm{(x^k, s^k)}_1 \le C_1 n \mu_k.
    \label{upper_bound_norm_x_s}
  \end{equation}
\end{lemma}
The proof below is based on the proof of \cite[Lemma~6.3]{wright1997primal}.

\begin{proof}
  From the definition of $\Neighborhood(\gamma_1, \gamma_2)$ in \eqref{def_neighborhood},
  we know
  $$
    \frac{\norm{(r_b(x^k), r_c(y^k, s^k))}}{\mu_k} \le \gamma_2 \frac{\norm{(r_b(x^0), r_c(y^0, s^0))}}{\mu_0},
  $$
  which implies
  \begin{equation}
    \mu_k \ge \frac{\norm{(r_b(x^k), r_c(y^k, s^k))}}{\gamma_2 \norm{(r_b(x^0), r_c(y^0, s^0))}} \mu_0 = \frac{\nu_k}{\gamma_2} \mu_0
    \label{mu_decreasing_lower_bound}
  \end{equation}
  from \eqref{residuals_decreasing} and $\gamma_2 \ge 1$.
  When we set
  $$
    (\bar{x}, \bar{y}, \bar{s}) = \nu_k (x^0, y^0, s^0) + (1 - \nu_k) (x^*, y^*, s^*) - (x^k, y^k, s^k),
  $$
  we have $A \bar{x} = 0$ and $A^\top \bar{y} + \bar{s} = 0$ from \eqref{residuals_decreasing} and \eqref{KKT_conditions},
  then
  \begin{align*}
    0 = & \ \bar{x}^\top \bar{s}                                                                                                                    \\
    =   & \ (\nu_k x^0 + (1 - \nu_k) x^* - x^k)^\top (\nu_k s^0 + (1 - \nu_k) s^* - s^k)                                                            \\
    =   & \ \nu_k^2 (x^0)^\top s^0 + \nu_k (1 - \nu_k) \left((x^0)^\top s^* + (x^*)^\top s^0\right) + (x^k)^\top s^k + (1 - \nu_k)^2 (x^*)^\top s^* \\
        & \ \ - \left(\nu_k ((x^0)^\top s^k + (s^0)^\top x^k) + (1 - \nu_k) ((x^k)^\top s^* + (s^k)^\top x^*)\right)
  \end{align*}
  is satisfied.
  Since all the components of $x^k, s^k, x^*, s^*$ are nonnegative,
  we have $((x^k)^\top s^* + (s^k)^\top x^*) \ge 0$.
  In addition, we have $(x^*)^\top s^* = 0$ from \eqref{KKT_conditions}.
  By using these and rearranging,
  we obtain
  \begin{align}
    \nu_k ((x^0)^\top s^k + (s^0)^\top x^k)
     & \le \nu_k^2 (x^0)^\top s^0 + \nu_k (1 - \nu_k) \left((x^0)^\top s^* + (x^*)^\top s^0\right) + (x^k)^\top s^k \notag                 \\
    [\because \eqref{def_mu}] \quad
     & = \nu_k^2 n \mu_0 + \nu_k (1 - \nu_k) \left((x^0)^\top s^* + (x^*)^\top s^0\right) + n \mu_k \notag                                 \\
    [\because \eqref{mu_decreasing_lower_bound}] \quad
     & \le \nu_k \gamma_2 n \mu_k + \gamma_2 (1 - \nu_k) \frac{\left((x^0)^\top s^* + (x^*)^\top s^0\right)}{\mu_0} \mu_k + n \mu_k \notag \\
    [\because \nu_k \in [0, 1]] \quad
     & \le (1 + \gamma_2) n \mu_k + \gamma_2 \frac{\left((x^0)^\top s^* + (x^*)^\top s^0\right)}{\mu_0} \mu_k.
    \label{x_s_upper}
  \end{align}
  Defining a constant $\xi$ by
  \begin{equation}
    \xi = \min_{i=1,2,\dots,n} \min(x_i^0, s_i^0) > 0,
    \label{def_xi}
  \end{equation}
  we have $(x^0)^\top s^k + (s^0)^\top x^k \ge \xi \norm{(x^k, s^k)}_1$.
  Therefore,
  from \eqref{x_s_upper},
  we obtain
  $$
    \nu_k \norm{(x^k, s^k)}_1 \le \xi^\inverse \left(1 + \gamma_2 + \gamma_2 \frac{(x^0)^\top s^* + (x^*)^\top s^0}{(x^0)^\top s^0}\right) n \mu_k.
  $$

  From \eqref{def_omega} and \eqref{def_initial_point},
  we have
  \begin{align*}
    \xi = \omega, \quad
    (x^0)^\top s^* + (x^*)^\top s^0  \le \norm{s^*}_\infty n \omega + \norm{x^*}_\infty n \omega = 2 n \omega^2, \quad
    (x^0)^\top s^0 = n \omega^2.
  \end{align*}
  We complete this proof by settings
  \begin{equation}
    C_1 = \omega^\inverse \left(1 + \gamma_2 + \gamma_2 \frac{2 n \omega^2}{n \omega^2}\right) = \frac{1 + 3 \gamma_2}{\omega}
    \label{def_C1}
  \end{equation}
  in \eqref{upper_bound_norm_x_s}.
  Since $\omega \ge 1$ in Assumption~\ref{assumption_opt_sol_bounded}, we have $C_1 \le 1 + 3 \gamma_2$,
  which implies an upper bound $1+3 \gamma_2$ is independent of $n$.
\end{proof}

For the proof of Lemma~\ref{lemma_first_derivative_upper},
the following lemma evaluates $\norm{S^k v_i^k}_\infty$.
\begin{lemma}[{\cite[Lemma~4.2]{mohammadisiahroudi2024efficient}}]
  \label{lemma_upper_derivatives_residual}
  For the derivatives $(\tilde{\dot{x}}, \tilde{\dot{y}}, \tilde{\dot{s}})$ and $(\tilde{\ddot{x}}, \tilde{\ddot{y}}, \tilde{\ddot{s}})$,
  when the residuals $\hat{r}_i^k$ satisfy \eqref{def_upper_derivatives_residual_MNES},
  it holds that
  \begin{equation}
    \norm{S^k v_i^k}_\infty \le \eta \mu_k.
    \label{upper_residual_term_MNES}
  \end{equation}
\end{lemma}
Then,
the following lemma holds similarly to \cite[Lemma~6.5]{wright1997primal} and \cite[Lemma~4.6]{mohammadisiahroudi2024efficient}.
\begin{lemma}
  There is a positive constant $C_2$ such that
  $$
    \max\left\{\norm{(D^k)^\inverse \tilde{\dot{x}}}, \norm{D^k \tilde{\dot{s}}}\right\} \le C_2 n \sqrt{\mu_k}
  $$
  \label{lemma_first_derivative_upper}
\end{lemma}
\begin{proof}
  Let
  $$
    (\bar{x}, \bar{y}, \bar{s}) = (\tilde{\dot{x}}, \tilde{\dot{y}}, \tilde{\dot{s}}) - \nu_k (x^0, y^0, s^0) + \nu_k (x^*, y^*, s^*).
  $$
  From \eqref{inexact_first_derivative_MNES_main_residual},
  \eqref{inexact_first_derivative_MNES_dual_residual},
  \eqref{residuals_decreasing} and \eqref{KKT_conditions},
  we have $A \bar{x} = 0$ and $A^\top \bar{y} + \bar{s} = 0$,
  therefore, $\bar{x}^\top \bar{s} = 0$.
  Thus, we obtain
  \begin{equation}
    \norm{(D^k)^\inverse \bar{x} + D^k \bar{s}}^2 = \norm{(D^k)^\inverse (\tilde{\dot{x}} - \nu_k (x^0 - x^*))}^2 + \norm{D^k (\tilde{\dot{s}} - \nu_k (s^0 - s^*))}^2.
    \label{eq_norm_D_inv_bar_x_plus_D_bar_s}
  \end{equation}
  From \eqref{inexact_first_derivative_MNES_duality}, it holds that
  \begin{align*}
    S^k \bar{x} + X^k \bar{s} & = (S^k \tilde{\dot{x}} + X^k \tilde{\dot{s}}) - \nu_k S^k (x^0 - x^*) - \nu_k X^k (s^0 - s^*) \\
                              & = (X^k s^k - \sigma \mu_k e - S^k v_1^k) - \nu_k S^k (x^0 - x^*) - \nu_k X^k (s^0 - s^*).
  \end{align*}
  Consequently,
  we verify
  \begin{equation}
    (D^k)^\inverse \bar{x} + D^k \bar{s} = (X^k S^k)^{-\frac{1}{2}} (X^k s^k - \sigma \mu_k e - S^k v_1^k) - \nu_k (D^k)^\inverse (x^0 - x^*) - \nu_k D^k (s^0 - s^*).
    \label{eq_D_inv_bar_x_plus_D_bar_s}
  \end{equation}
  For any vector $a \in \Real^d$,
  \begin{equation}
    \norm{a}_1 \le \sqrt{n} \norm{a} \le n \norm{a}_\infty
    \label{inequality_norms}
  \end{equation}
  holds from \cite[Lemma~3.1]{yang2020arc}.
  From \eqref{eq_norm_D_inv_bar_x_plus_D_bar_s}, \eqref{eq_D_inv_bar_x_plus_D_bar_s},
  \eqref{inequality_norms} and Lemma~\ref{lemma_upper_derivatives_residual},
  we obtain
  \begin{align}
     & \norm{(D^k)^\inverse (\tilde{\dot{x}} - \nu_k (x^0 - x^*))}^2 + \norm{D^k (\tilde{\dot{s}} - \nu_k (s^0 - s^*))}^2 \notag                                                                               \\
     & = \norm{(X^k S^k)^{-\frac{1}{2}} (X^k s^k - \sigma \mu_k e - S^k v_1^k) - \nu_k (D^k)^\inverse (x^0 - x^*) - \nu_k D^k (s^0 - s^*)}^2 \notag                                                            \\
     & \le \left\{\norm{X^k S^k}^{-\frac{1}{2}} \left(\norm{X^k s^k - \sigma \mu_k e} + \norm{S^k v_1^k}\right) + \nu_k \norm{(D^k)^\inverse (x^0 - x^*)} + \nu_k \norm{D^k (s^0 - s^*)}\right\}^2 \notag      \\
     & \le \left\{\norm{X^k S^k}^{-\frac{1}{2}} \left(\norm{X^k s^k - \sigma \mu_k e} + \sqrt{n} \eta \mu_k \right) + \nu_k \left(\norm{(D^k)^\inverse (x^0 - x^*)} + \norm{D^k (s^0 - s^*)}\right)\right\}^2.
    \label{upper_sum_of_norm_of_D_inv_bar_x_plus_D_bar_s}
  \end{align}
  In addition, $x^k_i s^k_i \ge \gamma \mu_k$ in \eqref{def_neighborhood} implies
  \begin{equation}
    \norm{X^k S^k}^{-\frac{1}{2}} \le \frac{1}{\sqrt{\gamma_1 \mu_k}} \label{upper_x_s_half_inverse}.
  \end{equation}
  From \eqref{upper_bound_norm_x_s} and \eqref{upper_x_s_half_inverse}, we have
  \begin{equation}
    \nu_k \norm{(x^k, s^k)}_1 \norm{(X S)^{-1 / 2}}
    \le \frac{C_1 n \sqrt{\mu_k}}{\sqrt{\gamma_1}}.
    \label{upper_xs_XS_half_inv}
  \end{equation}
  According to the derivation in \cite[Lemma~6.5]{wright1997primal},
  we have
  \begin{align}
     & \norm{X^k s^k - \sigma \mu_k e} \le n \mu_k, \label{upper_X_s}                                            \\
     & \nu_k \left(\norm{(D^k)^\inverse (x^0-x^*)} + \norm{D^k (s^0-s^*)}\right) \notag                          \\
     & \ \le \nu_k \norm{(x^k, s^k)}_1 \norm{(X S)^{-1 / 2}} \max \left\{\norm{x^0-x^*}, \norm{s^0-s^*}\right\}.
    \label{upper_nu_k_norm}
  \end{align}
  Therefore,
  from \eqref{upper_nu_k_norm} and \eqref{upper_xs_XS_half_inv},
  we obtain
  \begin{align}
     & \nu_k \left(\norm{(D^k)^\inverse (x^0-x^*)} + \norm{D^k (s^0-s^*)}\right) \notag                      \\
     & \ \le \frac{C_1}{\sqrt{\gamma_1}} n \sqrt{\mu_k } \max \left\{\norm{x^0-x^*}, \norm{s^0-s^*}\right\}.
    \label{upper_sum_of_norm_of_D_xs_0_minus_xs_star}
  \end{align}
  Therefore,
  we have
  \begin{align*}
    \norm{(D^k)^\inverse \tilde{\dot{x}}}
    \le & \ \norm{(D^k)^\inverse (\tilde{\dot{x}} - \nu_k (x^0 - x^*))} + \nu_k \norm{(D^k)^\inverse (x^0 - x^*)}                                                                 \\
    [\because \eqref{upper_sum_of_norm_of_D_inv_bar_x_plus_D_bar_s}] \quad
    \le & \ \norm{X^k S^k}^{-\frac{1}{2}} \left(\norm{X^k s^k - \sigma \mu_k e} + \sqrt{n} \eta \mu_k \right)                                                                     \\
        & + 2 \nu_k \left(\norm{(D^k)^\inverse (x^0 - x^*)} + \norm{D^k (s^0 - s^*)}\right)                                                                                       \\
    [\because \eqref{upper_x_s_half_inverse}, \eqref{upper_X_s}] \quad
    \le & \ \frac{\sqrt{\mu_k}}{\sqrt{\gamma_1}} \left(n + \sqrt{n} \eta \right)+ 2 \nu_k \left(\norm{(D^k)^\inverse (x^0 - x^*)} + \norm{D^k (s^0 - s^*)}\right)                 \\
    [\because \eqref{upper_sum_of_norm_of_D_xs_0_minus_xs_star}] \quad
    \le & \ \frac{\sqrt{\mu_k}}{\sqrt{\gamma_1}} \left(n + \sqrt{n} \eta \right)+ \frac{2 C_1 n \sqrt{\mu_k}}{\sqrt{\gamma_1}} \max \left\{\norm{x^0-x^*}, \norm{s^0-s^*}\right\} \\
    \le & \ \frac{1}{\sqrt{\gamma_1}} \left(1 + \eta + 2 C_1 \max \left\{\norm{x^0-x^*}, \norm{s^0-s^*}\right\}\right) n \sqrt{\mu_k}.
  \end{align*}
  Since the optimal set is bounded from Assumption~\ref{assumption_opt_sol_bounded},
  \begin{equation}
    C_2 := \gamma_1^{-1/2} \left(1 + \eta + 2 C_1 \max \left\{\norm{x^0-x^*}, \norm{s^0-s^*}\right\}\right)
    \label{def_C2}
  \end{equation}
  is also bounded,
  and we can prove this lemma by setting this $C_2$.
  We can similarly show $\tilde{\dot{s}} \le C_2 n \sqrt{\mu_k}$.
\end{proof}

From Lemma~\ref{lemma_first_derivative_upper}, it holds that
\begin{equation}
  \norm{\tilde{\dot{x}} \circ \tilde{\dot{s}}} \le \norm{(D^k)^\inverse \tilde{\dot{x}}} \norm{D^k \tilde{\dot{s}}} \le C_2^2 n^2 \mu_k.
  \label{upper_first_derivative_Hadamard}
\end{equation}
Similarly,
we evaluate the terms related to $G_i^k(\alpha)$, $g^k(\alpha)$ and $h^k(\alpha)$.

\begin{lemma}
  \label{lemma_upper_of_first_and_second_derivatives}
  There are positive constants $C_3$ and $C_4$ such that
  \begin{align*}
    \norm{\tilde{\ddot{x}} \circ \tilde{\ddot{s}}}                                                                   & \le C_3 n^4 \mu_k,        \\
    \max\left\{\norm{(D^k)^\inverse \tilde{\ddot{x}}}, \norm{D^k \tilde{\ddot{s}}}\right\}                           & \le C_4 n^2 \sqrt{\mu_k}, \\
    \max \left\{\norm{\tilde{\ddot{x}} \circ \tilde{\dot{s}}}, \norm{\tilde{\dot{x}} \circ \tilde{\ddot{s}}}\right\} & \le C_2 C_4 n^3 \mu_k.
  \end{align*}
\end{lemma}

\begin{proof}
  When $u^\top v \ge 0$ for any vector pairs of $u, v$, the inequality
  $$
    \norm{u \circ v} \leq 2^{-\frac{3}{2}}\norm{u + v}^2
  $$
  holds from \citet[Lemma~5.3]{wright1997primal},
  so the following is satisfied:
  $$
    \norm{\tilde{\ddot{x}} \circ \tilde{\ddot{s}}}
    = \norm{(D^k)^\inverse \tilde{\ddot{x}} \circ D^k \tilde{\ddot{s}}}
    \le 2^{-\frac{3}{2}} \norm{(D^k)^\inverse \tilde{\ddot{x}} + D^k \tilde{\ddot{s}}}^2.
  $$
  From $(D^k)^\inverse \tilde{\ddot{x}} + D^k \tilde{\ddot{s}} = (X^k S^k)^{-1/2} (S^k \tilde{\ddot{x}} + X^k \tilde{\ddot{s}})$,
  \begin{align}
    \norm{(D^k)^\inverse \tilde{\ddot{x}} + D^k \tilde{\ddot{s}}}
     & \le \norm{X^k S^k}^{-\frac{1}{2}} \norm{S^k \tilde{\ddot{x}} + X^k \tilde{\ddot{s}}} \notag                             \\
    [\because \eqref{inexact_second_derivative_duality}] \quad
     & \le \norm{X^k S^k}^{-\frac{1}{2}} \left(2 \norm{\tilde{\dot{x}} \circ \tilde{\dot{s}}} + \norm{S^k v_2^k}\right) \notag \\
    [\because \eqref{upper_x_s_half_inverse}, \eqref{upper_first_derivative_Hadamard}, \eqref{upper_residual_term_MNES}, \eqref{inequality_norms}] \quad
     & \le \frac{1}{\sqrt{\gamma_1 \mu_k}} \left(2 C_2^2 n^2 \mu_k + \sqrt{n} \eta \mu_k \right) \notag                        \\
     & \le \frac{\sqrt{\mu_k}}{\sqrt{\gamma_1}} (2 C_2^2 n^2 + \sqrt{n} \eta).
    \label{upper_norm_D_inv_ddot_x_plus_D_ddot_s}
  \end{align}
  From the above,
  we can obtain
  $$
    \norm{\tilde{\ddot{x}} \circ \tilde{\ddot{s}}}
    \le 2^{-\frac{3}{2}} \frac{\mu_k}{\gamma_1} (2 C_2^2 n^2 + \sqrt{n} \eta)^2
    \le \frac{(2 C_2^2 + \eta)^2}{2^{\frac{3}{2}} \gamma_1} n^4 \mu_k
    =: C_3 n^4 \mu_k.
  $$
  From \eqref{inexact_second_derivative_main_residual} and \eqref{inexact_second_derivative_dual_residual},
  we know
  \begin{equation}
    \tilde{\ddot{x}}^\top \tilde{\dot{s}} = 0,
    \label{second_derivative_x_s_zero_inner_product}
  \end{equation}
  then \eqref{upper_norm_D_inv_ddot_x_plus_D_ddot_s} leads to
  \begin{align*}
    \max\left\{\norm{(D^k)^\inverse \tilde{\ddot{x}}}^2, \norm{D^k \tilde{\ddot{s}}}^2\right\}
     & \le \norm{(D^k)^\inverse \tilde{\ddot{x}} + D^k \tilde{\ddot{s}}}^2    \\
     & \le \frac{\mu_k}{\gamma_1} (2 C_2^2 n^2 + \sqrt{n} \eta)^2             \\
     & \le \frac{\mu_k}{\gamma_1} (2 C_2^2 +  \eta)^2 n^4 =: C_4^2 n^4 \mu_k,
  \end{align*}
  $$
    \norm{\tilde{\ddot{x}} \circ \tilde{\dot{s}}}
    \le \norm{(D^k)^\inverse \tilde{\ddot{x}}} \norm{D^k \tilde{\dot{s}}}
    \le C_4 n^2 \sqrt{\mu_k} C_2 n \sqrt{\mu_k}
    = C_2 C_4 n^3 \mu_k.
  $$
  We can show the inequality of $\norm{\tilde{\dot{x}} \circ \tilde{\ddot{s}}}$ similarly.
\end{proof}

Using these lemmas,
we are ready to prove Proposition~\ref{proposition_lower_bound_of_step_size}.

\begin{proof}[Proof of Proposition~\ref{proposition_lower_bound_of_step_size}]
  Firstly,
  we derive the equations necessary for the proofs.
  We have the following simple identity:
  \begin{equation}
    -2(1 - \cos(\alpha)) + \sin^2(\alpha) = -(1-\cos(\alpha))^2.
    \label{sin_cos_1}
  \end{equation}
  Therefore,
  we can obtain
  \begin{align}
    x^k(\alpha) \circ s^k(\alpha)
    = & \ \left(x^k-\tilde{\dot{x}}\sin(\alpha)+\tilde{\ddot{x}}(1-\cos(\alpha))\right) \circ \left(s^k-\tilde{\dot{s}}\sin(\alpha)+\tilde{\ddot{s}}(1-\cos(\alpha))\right) \notag \\
    = & \ x^k \circ s^k - \left(x^k \circ \tilde{\dot{s}} + \tilde{\dot{x}} \circ s^k\right)\sin(\alpha)
    + \left(x^k \circ \tilde{\ddot{s}} + \tilde{\ddot{x}} \circ s^k\right)(1 - \cos(\alpha)) \notag                                                                                \\
      & \ + \tilde{\dot{x}} \circ \tilde{\dot{s}} \sin^2 (\alpha)
    - \left(\tilde{\dot{x}} \circ \tilde{\ddot{s}} + \tilde{\ddot{x}} \circ \tilde{\dot{s}}\right) \sin(\alpha) (1 - \cos(\alpha))
    + \tilde{\ddot{x}} \circ \tilde{\ddot{s}} (1 - \cos(\alpha))^2 \notag                                                                                                          \\
    [\because \eqref{inexact_first_derivative_MNES_duality}, \eqref{inexact_second_derivative_duality}] \quad =
      & \ x^k \circ s^k - (x^k \circ s^k - \sigma \mu_k e - S^k v_1^k) \sin(\alpha) + \left(-2 \tilde{\dot{x}} \circ \tilde{\dot{s}} - S^k v_2^k \right)(1 - \cos(\alpha)) \notag  \\
      & \ + \tilde{\dot{x}} \circ \tilde{\dot{s}} \sin^2 (\alpha)
    - \left(\tilde{\dot{x}} \circ \tilde{\ddot{s}} + \tilde{\ddot{x}} \circ \tilde{\dot{s}}\right) \sin(\alpha) (1 - \cos(\alpha))
    + \tilde{\ddot{x}} \circ \tilde{\ddot{s}} (1 - \cos(\alpha))^2 \notag                                                                                                          \\
    [\because \eqref{sin_cos_1}] \quad =
      & \ x^k \circ s^k (1 - \sin(\alpha)) + \sigma \mu_k \sin(\alpha) e \notag                                                                                                    \\
      & \ + \left(\tilde{\ddot{x}} \circ \tilde{\ddot{s}} - \tilde{\dot{x}} \circ \tilde{\dot{s}}\right) (1 - \cos(\alpha))^2
    - \left(\tilde{\dot{x}} \circ \tilde{\ddot{s}} + \tilde{\ddot{x}} \circ \tilde{\dot{s}}\right) \sin(\alpha) (1 - \cos(\alpha)) \notag                                          \\
      & \ + S^k v_1^k \sin(\alpha) - S^k v_2^k (1 - \cos(\alpha))
    \label{x_s_alpha_Hadamard}
  \end{align}
  and
  \begin{align}
    x^k(\alpha)^\top s^k(\alpha)
    = & \ \left(x^k-\tilde{\dot{x}}\sin(\alpha)+\tilde{\ddot{x}}(1-\cos(\alpha))\right)^\top \left(s^k-\tilde{\dot{s}}\sin(\alpha)+\tilde{\ddot{s}}(1-\cos(\alpha))\right) \notag \\
    [\because \eqref{x_s_alpha_Hadamard}, \eqref{def_mu}, \eqref{second_derivative_x_s_zero_inner_product}] \quad
    = & \ (x^k)^\top s^k \left((1 - \sin(\alpha)) + \sigma \sin(\alpha)\right) \notag                                                                                             \\
      & \ - \tilde{\dot{x}}^\top \tilde{\dot{s}} (1 - \cos(\alpha))^2
    - \left(\tilde{\dot{x}}^\top \tilde{\ddot{s}} + \tilde{\ddot{x}}^\top \tilde{\dot{s}}\right) \sin(\alpha) (1 - \cos(\alpha)) \notag                                           \\
      & \ + \sin(\alpha) \sum_{i=1}^n [S^k v_1^k]_i - (1 - \cos(\alpha)) \sum_{i=1}^n [S^k v_2^k]_i.
    \label{x_s_alpha_inner_product}
  \end{align}

  From Lemmas~\ref{lemma_first_derivative_upper} and \ref{lemma_upper_of_first_and_second_derivatives} and the Cauchy-Schwartz inequality,
  we know
  \begin{subequations}
    \begin{align}
      \abs{\tilde{\dot{x}}_i \tilde{\dot{s}}_i}, \, \abs{\tilde{\dot{x}}^\top \tilde{\dot{s}}}
                                                  & \le \norm{(D^k)^\inverse \tilde{\dot{x}}} \norm{D^k \tilde{\dot{s}}} \le C_2^2 n^2 \mu_k
      \label{upper_product_of_dot_x_and_dot_s_element_wise}                                                                                     \\
      \abs{\tilde{\ddot{x}}_i \tilde{\dot{s}}_i}, \, \abs{\tilde{\ddot{x}}^\top \tilde{\dot{s}}}
                                                  & \le \norm{(D^k)^\inverse \tilde{\ddot{x}}} \norm{D^k \tilde{\dot{s}}} \le C_2 C_4 n^3 \mu_k \\
      \abs{\tilde{\dot{x}}_i \tilde{\ddot{s}}_i}, \, \abs{\tilde{\dot{x}}^\top \tilde{\ddot{s}}}
                                                  & \le \norm{(D^k)^\inverse \tilde{\dot{x}}} \norm{D^k \tilde{\ddot{s}}} \le C_2 C_4 n^3 \mu_k \\
      \abs{\tilde{\ddot{x}}_i \tilde{\ddot{s}}_i} & \le \norm{(D^k)^\inverse \tilde{\ddot{x}}} \norm{D^k \tilde{\ddot{s}}} \le C_4^2 n^4 \mu_k
      \label{upper_product_of_ddot_x_and_ddot_s_element_wise}
    \end{align}
    \label{uppers_product_of_derivatives}
  \end{subequations}
  Here,
  $\abs{\tilde{\ddot{x}}^\top \tilde{\ddot{s}}} = 0$ holds due to \eqref{second_derivative_x_s_zero_inner_product}.
  Furthermore, we have
  \begin{equation}
    \sin^2(\alpha) = 1 - \cos^2(\alpha) \ge 1 - \cos(\alpha)
    \label{sin_square_more_than_one_minus_cos}
  \end{equation}
  from $\alpha \in (0, \pi/2]$.

  We prove that the step size $\alpha$ satisfying $g^k(\alpha) \ge 0$ is bounded away from zero.
  From \eqref{x_s_alpha_inner_product},
  \begin{align}
    x^k(\alpha)^\top s^k(\alpha)
    \ge & \ (x^k)^\top s^k \left((1 - \sin(\alpha)) + \sigma \sin(\alpha)\right) \notag                                                             \\
        & \ - \abs{\tilde{\dot{x}}^\top \tilde{\dot{s}}} (1 - \cos(\alpha))^2
    - \left(\abs{\tilde{\dot{x}}^\top \tilde{\ddot{s}}} + \abs{\tilde{\ddot{x}}^\top \tilde{\dot{s}}}\right) \sin(\alpha) (1 - \cos(\alpha)) \notag \\
        & \ - \norm{S^k v_1^k}_1 \sin(\alpha) - \norm{S^k v_2^k}_1 (1 - \cos(\alpha)) \notag                                                        \\
    [\because \eqref{inequality_norms}, \eqref{upper_residual_term_MNES}] \quad
    \ge & \ (x^k)^\top s^k \left((1 - \sin(\alpha)) + \sigma \sin(\alpha)\right) \notag                                                             \\
        & \ - \abs{\tilde{\dot{x}}^\top \tilde{\dot{s}}} (1 - \cos(\alpha))^2
    - \left(\abs{\tilde{\dot{x}}^\top \tilde{\ddot{s}}} + \abs{\tilde{\ddot{x}}^\top \tilde{\dot{s}}}\right) \sin(\alpha) (1 - \cos(\alpha)) \notag \\
        & \ - \eta n \mu_k (\sin(\alpha) + 1 - \cos(\alpha)).
    \label{lower_x_s_alpha_inner_product}
  \end{align}
  Therefore,
  \begin{align*}
    g^k(\alpha) = & \ x^k(\alpha)^\top s^k(\alpha) - (1 - \sin(\alpha))(x^k)^\top s^k                                                                 \\
    [\because \eqref{lower_x_s_alpha_inner_product}] \quad
    \ge           & \ \sigma (x^k)^\top s^k \sin(\alpha) - \eta n \mu_k \left(\sin(\alpha) + 1 -\cos(\alpha)\right)                                   \\
                  & \ - \abs{\tilde{\dot{x}}^\top \tilde{\dot{s}}} (1 - \cos(\alpha))^2
    - \left(\abs{\tilde{\dot{x}}^\top \tilde{\ddot{s}}} + \abs{\tilde{\ddot{x}}^\top \tilde{\dot{s}}}\right) \sin(\alpha) (1 - \cos(\alpha))          \\
    [\because \eqref{def_mu}, \eqref{sin_square_more_than_one_minus_cos}] \quad
    \ge           & \ \sigma n \mu_k \sin(\alpha) - \eta n \mu_k \left(\sin(\alpha) + \sin^2 (\alpha)\right)                                          \\
                  & \ - \abs{\tilde{\dot{x}}^\top \tilde{\dot{s}}} \sin^4 (\alpha)
    - \left(\abs{\tilde{\dot{x}}^\top \tilde{\ddot{s}}} + \abs{\tilde{\ddot{x}}^\top \tilde{\dot{s}}}\right) \sin^3(\alpha)                           \\
    [\because \eqref{uppers_product_of_derivatives}] \quad
    \ge           & \ n \mu_k \sin(\alpha) \left((\sigma - \eta) - \eta \sin(\alpha) - C_2^2 n \sin^3 (\alpha) - 2 C_2 C_4 n^2 \sin^2(\alpha)\right).
  \end{align*}
  Since $\left(-\eta \sin(\alpha) - C_2^2 n \sin^3 (\alpha) - 2 C_2 C_4 n^2 \sin^2 (\alpha)\right)$ is monotonically decreasing
  and $\sigma > \eta$ holds from \eqref{parameter_condition_for_G_i} and $\gamma_1 \in (0, 1)$,
  there exists the step size $\hat{\alpha}_1 \in (0, \pi / 2]$ satisfying the last formula of the right-hand side is no less than 0.
  When
  $$
    \sin(\hat{\alpha}_1) \le \frac{\sigma - \eta}{2 n} \frac{1}{\max \left\{\eta, C_2^\frac{2}{3}, \sqrt{2 C_2 C_4}\right\}},
  $$
  from $0 < \sigma -\eta < \sigma \le 1$,
  \begin{align*}
     & (\sigma - \eta) - \eta \sin(\hat{\alpha}_1) - C_2^2 n \sin^3 (\hat{\alpha}_1) - 2 C_2 C_4 n^2 \sin^2 (\hat{\alpha}_1) \\
     & \quad \ge (\sigma - \eta) - \frac{\sigma - \eta}{2 n} - \frac{(\sigma - \eta)^3}{8 n^2} - \frac{(\sigma - \eta)^2}{4} \\
     & \quad \ge (\sigma - \eta) \left(1 - \frac{1}{2} - \frac{1}{8} - \frac{1}{4}\right) \quad \ge 0.
  \end{align*}
  Therefore,
  $g^k(\alpha) \ge 0$ is satisfied for any $\alpha \in (0, \hat{\alpha}_1]$.

  Next, we consider the range of $\alpha$ such that $G^k_i(\alpha) \ge 0$.
  From \eqref{uppers_product_of_derivatives},
  \begin{subequations}
    \label{upper_derivatives_element_wise_munus_products}
    \begin{align}
      \abs{\tilde{\dot{x}}_i \tilde{\dot{s}}_i - \frac{\gamma_1}{n}\tilde{\dot{x}}^\top \tilde{\dot{s}}}
       & \le \left(1 + \frac{\gamma_1}{n}\right) C_2^2 n^2 \mu_k \le 2 C_2^2 n^2 \mu_k \label{upper_derivatives_element_wise_minus_products_dot} \\
      \abs{\tilde{\ddot{x}}_i \tilde{\dot{s}}_i - \frac{\gamma_1}{n}\tilde{\ddot{x}}^\top \tilde{\dot{s}}}, \,
      \abs{\tilde{\dot{x}}_i \tilde{\ddot{s}}_i - \frac{\gamma_1}{n}\tilde{\dot{x}}^\top \tilde{\ddot{s}}}
       & \le 2 C_2 C_4 n^3 \mu_k
    \end{align}
  \end{subequations}
  is satisfied.
  Therefore, we have
  \begin{align*}
    G_i^k(\alpha) = & \ x_i^k(\alpha) s_i^k(\alpha) - \gamma_1 \mu_k(\alpha)                                                                                                                                                                                             \\
    [\because \eqref{x_s_alpha_Hadamard}, \eqref{def_mu}, \eqref{x_s_alpha_inner_product}] \quad
    \ge             & \ x^k_i s^k_i (1 - \sin(\alpha)) + \sigma \mu_k \sin(\alpha)                                                                                                                                                                                       \\
                    & \ + \left(\tilde{\ddot{x}}_i \tilde{\ddot{s}}_i - \tilde{\dot{x}}_i \tilde{\dot{s}}_i \right) (1 - \cos(\alpha))^2
    - \left(\tilde{\dot{x}}_i \tilde{\ddot{s}}_i + \tilde{\ddot{x}}_i \tilde{\dot{s}}_i \right) \sin(\alpha) (1 - \cos(\alpha))                                                                                                                                          \\
                    & \ - \norm{S^k v_1^k}_\infty \sin(\alpha) - \norm{S^k v_2^k}_\infty (1 - \cos(\alpha))                                                                                                                                                              \\
                    & \ - \frac{\gamma_1}{n} \bigg(n \mu_k \left((1 - \sin(\alpha)) + \sigma \sin(\alpha)\right)                                                                                                                                                         \\
                    & \ - \tilde{\dot{x}}^\top \tilde{\dot{s}} (1 - \cos(\alpha))^2 - \left(\tilde{\dot{x}}^\top \tilde{\ddot{s}} + \tilde{\ddot{x}}^\top \tilde{\dot{s}}\right) \sin(\alpha) (1 - \cos(\alpha))) \notag                                                 \\
                    & \ + \norm{S^k v_1^k}_1 \sin(\alpha) + \norm{S^k v_2^k}_1 (1 - \cos(\alpha)) \bigg)                                                                                                                                                                 \\
    [\because \eqref{def_neighborhood}, \eqref{upper_residual_term_MNES}, \eqref{inequality_norms}] \quad
    \ge             & \ (1 - \gamma_1) \sigma \mu_k \sin(\alpha) - (1 + \gamma_1) \eta \mu_k (\sin(\alpha) + 1 - \cos(\alpha))                                                                                                                                           \\
                    & \ + \tilde{\ddot{x}}_i \tilde{\ddot{s}}_i (1 - \cos(\alpha))^2
    - \left(\tilde{\dot{x}}_i \tilde{\dot{s}}_i - \frac{\gamma_1}{n} \tilde{\dot{x}}^\top \tilde{\dot{s}}\right) (1 - \cos(\alpha))^2                                                                                                                                    \\
                    & \ - \left(\tilde{\dot{x}}_i \tilde{\ddot{s}}_i - \frac{\gamma_1}{n} \tilde{\dot{x}}^\top \tilde{\ddot{s}} + \tilde{\ddot{x}}_i \tilde{\dot{s}}_i - \frac{\gamma_1}{n} \tilde{\ddot{x}}^\top \tilde{\dot{s}}\right) \sin(\alpha) (1 - \cos(\alpha)) \\
    [\because \eqref{sin_square_more_than_one_minus_cos}, \eqref{upper_product_of_ddot_x_and_ddot_s_element_wise}, \eqref{upper_derivatives_element_wise_munus_products}] \quad
    \ge             & \ \mu_k \sin(\alpha) \bigg((1 - \gamma_1) \sigma - (1 + \gamma_1) \eta - (1 + \gamma_1) \eta \sin(\alpha)                                                                                                                                          \\
                    & \ - (C_4^2 n^4 + 2 C_2^2 n^2) \sin^3(\alpha) - 4 C_2 C_4 n^3 \sin^2 (\alpha)\bigg).
  \end{align*}
  We can derive the same discussion as $g^k(\alpha)$ using \eqref{parameter_condition_for_G_i}.
  When
  $$
    \sin(\hat{\alpha}_2) \le \frac{(1 - \gamma_1) \sigma - (1 + \gamma_1) \eta}{2 n^\frac{3}{2}} \frac{1}{\max \left\{(1 + \gamma_1) \eta, (C_4^2 + 2 C_2^2)^\frac{1}{3}, 2 \sqrt{C_2 C_4}\right\}},
  $$
  from $0 < (1 - \gamma_1) \sigma - (1 + \gamma_1) \eta < \sigma \le 1$,
  \begin{align*}
     & (1 - \gamma_1) \sigma - (1 + \gamma_1) \eta - (1 + \gamma_1) \eta \sin(\hat{\alpha}_2) - (C_4^2 n^4 + 2 C_2^2 n^2) \sin^3(\hat{\alpha}_2) - 4 C_2 C_4 n^3 \sin^2 (\hat{\alpha}_2) \\
     & \quad \ge ((1 - \gamma_1) \sigma - (1 + \gamma_1) \eta) \left(1 - \frac{1}{2 n^\frac{3}{2}} - \frac{1}{2^3 n^\frac{1}{2}} - \frac{1}{2^2}\right)                                  \\
     & \quad \ge ((1 - \gamma_1) \sigma - (1 + \gamma_1) \eta) \left(1 - \frac{1}{2} - \frac{1}{8} - \frac{1}{4}\right)                                                                  \\
     & \quad \ge 0.
  \end{align*}
  Therefore,
  $G_i^k(\alpha) \ge 0$ is satisfied for $\alpha \in (0, \hat{\alpha}_2]$.

  Lastly, we consider $h^k(\alpha) \ge 0$.
  Similarly to the derivation of \eqref{lower_x_s_alpha_inner_product},
  we can obtain the following:
  \begin{align}
    x^k(\alpha)^\top s^k(\alpha)
    \le & \ (x^k)^\top s^k \left((1 - \sin(\alpha)) + \sigma \sin(\alpha)\right) \notag                                                             \\
        & \ + \abs{\tilde{\dot{x}}^\top \tilde{\dot{s}}} (1 - \cos(\alpha))^2
    + \left(\abs{\tilde{\dot{x}}^\top \tilde{\ddot{s}}} + \abs{\tilde{\ddot{x}}^\top \tilde{\dot{s}}}\right) \sin(\alpha) (1 - \cos(\alpha)) \notag \\
        & \ + \eta n \mu_k (\sin(\alpha) + 1 - \cos(\alpha)),
    \label{upper_x_s_alpha_inner_product}
  \end{align}
  Therefore,
  \begin{align*}
    h^k(\alpha) = & \ \left(1 - (1-\beta) \sin(\alpha)\right)(x^k)^\top s^k - x^k(\alpha)^\top s^k(\alpha)                                                    \\
    [\because \eqref{upper_x_s_alpha_inner_product}] \quad
    \ge           & \ (x^k)^\top s^k \left(\beta \sin(\alpha) - \sigma \sin(\alpha)\right) - \eta n \mu_k (\sin(\alpha) + 1 - \cos(\alpha))                   \\
                  & \ - \abs{\tilde{\dot{x}}^\top \tilde{\dot{s}}} (1 - \cos(\alpha))^2
    - \left(\abs{\tilde{\dot{x}}^\top \tilde{\ddot{s}}} + \abs{\tilde{\ddot{x}}^\top \tilde{\dot{s}}}\right) \sin(\alpha) (1 - \cos(\alpha))                  \\
    [\because \eqref{def_mu}] \quad
    =             & \ n \mu_k \left(\beta \sin(\alpha) - \sigma \sin(\alpha) - \eta (\sin(\alpha) + 1 - \cos(\alpha))\right)                                  \\
                  & \ - \abs{\tilde{\dot{x}}^\top \tilde{\dot{s}}} (1 - \cos(\alpha))^2
    - \left(\abs{\tilde{\dot{x}}^\top \tilde{\ddot{s}}} + \abs{\tilde{\ddot{x}}^\top \tilde{\dot{s}}}\right) \sin(\alpha) (1 - \cos(\alpha))                  \\
    [\because \eqref{uppers_product_of_derivatives}] \quad
    \ge           & \ n \mu_k \left((\beta - \sigma - \eta) \sin(\alpha) - \eta (1 - \cos(\alpha))\right)                                                     \\
                  & \ - C_2^2 n^2 \mu_k (1 - \cos(\alpha))^2 - 2 C_2 C_4 n^3 \mu_k \sin(\alpha) (1 - \cos(\alpha))                                            \\
    [\because \eqref{sin_square_more_than_one_minus_cos}] \quad
    \ge           & \ n \mu_k \sin(\alpha) \left((\beta - \sigma - \eta) - \eta \sin(\alpha) - C_2^2 n \sin^3 (\alpha) - 2 C_2 C_4 n^2 \sin^2(\alpha)\right).
  \end{align*}
  The last coefficient on the right-hand side is cubic for $\sin(\alpha)$ and monotonically decreasing for $\alpha$.
  Therefore,
  it is possible to take a step size $\hat{\alpha}_3$ satisfying $h^k(\hat{\alpha}_3) \ge 0$ from \eqref{parameter_condition_beta_more_than_sigma_plus_eta}.
  When
  $$
    \sin(\hat{\alpha}_3) \le \frac{\beta - \sigma - \eta}{2 n} \frac{1}{\max \left\{\eta, C_2^\frac{2}{3}, \sqrt{2 C_2 C_4}\right\}},
  $$
  from $0 < \beta - \sigma -\eta < \beta < 1$, we know
  \begin{align*}
     & (\beta - \sigma - \eta) - \eta \sin(\hat{\alpha}_3) - C_2^2 n \sin^3 (\hat{\alpha}_3) - 2 C_2 C_4 n^2 \sin^2(\hat{\alpha}_3)                          \\
     & \quad \ge (\beta - \sigma - \eta) - \frac{\beta - \sigma - \eta}{2 n} - \frac{(\beta - \sigma - \eta)^3}{8 n^2} - \frac{(\beta - \sigma - \eta)^2}{4} \\
     & \quad > (\beta - \sigma - \eta) \left(1 - \frac{1}{2} - \frac{1}{8} - \frac{1}{4}\right)                                                              \\
     & \quad = \frac{\beta - \sigma - \eta}{8}
    > 0.
  \end{align*}
  Therefore,
  $g^k(\alpha) \ge 0$ is satisfied for $\alpha \in (0, \hat{\alpha}_3]$.

  From the above discussions,
  when $\hat{\alpha}$ is taken such that
  \begin{equation}
    \sin(\hat{\alpha}) = \frac{1}{n^\frac{3}{2}} \frac{
      \min \left\{(1 - \gamma_1) \sigma - (1 + \gamma_1) \eta, \beta - \sigma - \eta\right\}
    }{
      2 \max \left\{(1 + \gamma_1) \eta, (C_4^2 + 2 C_2^2)^\frac{1}{3}, 2 \sqrt{C_2 C_4}\right\}
    },
    \label{def_min_step_size}
  \end{equation}
  $g^k(\alpha), G_i^k(\alpha), h^k(\alpha) \ge 0$ are satisfied for all
  $k$ and $\alpha \in (0, \hat{\alpha}]$.
\end{proof}

Since $\hat{\alpha}$ defined in \eqref{def_min_step_size} can satisfy the conditions in line~\ref{line_algo_II_arc_search_decide_step_size} of Algorithm~\ref{algorithm_II_arc_IPM}, we can find the step length
$\alpha_k \ge \hat{\alpha} > 0$.
Therefore,
Algorithm~\ref{algorithm_II_arc_IPM} is well-defined.
From $h^k(\alpha_k) \ge 0$ for all $k$,
\begin{align}
  h^k(\alpha_k) \ge 0 \Rightarrow
  x^k(\alpha_k)^\top s^k(\alpha_k) & \le (1 - (1 - \beta)\sin(\alpha_k))(x^k)^\top s^k \nonumber     \\
                                   & \le (1 - (1 - \beta)\sin(\hat{\alpha}))(x^k)^\top s^k \nonumber \\
                                   & \le (1 - (1 - \beta)\sin(\hat{\alpha}))^k (x^0)^\top s^0.
  \label{mu-decrement}
\end{align}
Due to \eqref{residuals_decreasing}, it also holds that
\begin{equation}
  \norm*{(r_b(x^k), r_c(y^k, s^k))}
  \le \left(1-\sin(\hat{\alpha})\right)^k
  \norm*{(r_b(x^0), r_c(y^k, s^0))}.
  \label{constraints_residual_decrement}
\end{equation}

We can prove the polynomial complexity of the proposed method
based on the following theorem.
\begin{theorem}[{\cite[Theorem~1.4]{yang2020arc}}]
  \label{polynomiality_by_mu}
  Suppose that an algorithm for solving \eqref{KKT_conditions} generates a sequence of iterations that satisfies
  $$
    \mu_{k+1} \le \left(1-\frac{\delta}{n^\psi}\right) \mu_k, \quad k=0,1,2, \ldots,
  $$
  for some positive constants $\delta$ and $\psi$.
  Then there exists an index $K$ with
  $$
    K = \order*{n^\psi \log (\mu_0 / \zeta)}
  $$
  such that
  $$
    \mu_k \le \zeta \text { for } \forall k \ge K.
  $$
\end{theorem}
Applying \eqref{mu-decrement}, \eqref{def_neighborhood}, $(x^k, y^k, s^k) \in \Neighborhood(\gamma_1, \gamma_2)$,
\eqref{constraints_residual_decrement} and a result that $\sin(\hat{\alpha})$ is propositional to $n^{-1.5}$ in \eqref{def_min_step_size}
to this theorem,
we can obtain the following theorem.
\begin{theorem}
  \label{main-theorem}
  Algorithm~\ref{algorithm_II_arc_IPM} generates a $\zeta$-optimal solution in at most
  $$
    \order{n^{1.5} \log \left(\frac{\max \{\mu_0, \norm{r_b(x^0), r_c(y^0, s^0)}\}}{\zeta}\right)}
  $$
  iterations.
\end{theorem}

In the case that the input data is integral,
\citet{al2009convergence} and \citet{mohammadisiahroudi2024efficient} analyze that the iteration complexity of II-line is $\order*{n^2 L}$,
where $L$ is the binary length of the input data denoted as
$$
  L=m n+m+n+\sum_{i, j}\left\lceil\log \left(\abs{a_{i j}}+1\right)\right\rceil+\sum_i\left\lceil\log \left(\abs{c_i}+1\right)\right\rceil+\sum_j\left\lceil\log \left(\abs{b_j}+1\right)\right\rceil.
$$
Theorem~\ref{main-theorem} indicates that II-arc can reduce the iteration complexity from $n^2$ to $n^{1.5}$,
by a factor of $n^{0.5}$.
This reduction is mainly brought by the ellipsoidal approximation in the arc-search method.
\section{Numerical experiments}
\label{section_numerical_experiments}
In this section,
we describe the implementation and the numerical experiments of the proposed method.
The experiments were conducted on a Linux server with Opteron 4386 (3.10GHz), 16 cores, and 128GB RAM,
and the methods were implemented with Python~3.10.9.
The code is openly available for public use\footnote{\url{https://github.com/Research-Iida/Arc-search_IPM_LP_py}.}.

The main purpose of the numerical experiments presented in this paper
is to show the performance of the proposed
inexact arc-search method in comparison with an existing inexact line-search method.
Therefore,
the comparison is intentionally restricted to the existing inexact line-search method,
while a comprehensive comparison with state-of-the-art solvers is beyond the scope of this paper.
For completeness,
numerical results obtained by the state-of-the-art solver are reported in Appendix~\ref{section_appendix_SOTA_results}.

\subsection{Inexact infeasible line-search IPM}
\label{section_II_line}
In this section,
we describe in detail the algorithm of an inexact infeasible line-search IPM(II-line) that is used as a baseline for comparison with the proposed method.
The algorithm considered in this paper is based on \cite[Algorithm~1]{mohammadisiahroudi2024efficient}.

At each iteration $k$,
the method computes an inexact search direction $\tilde{\dot{z}}$ from \eqref{inexact_first_derivative_MNES},
and updates the next iteration point with the following formula:
\begin{subequations}
  \label{def_variable_alpha_with_inexact_derivatives_line}
  \begin{align}
    x^k(\alpha) = & \ x^k - \alpha \tilde{\dot{x}}, \\
    y^k(\alpha) = & \ y^k - \alpha \tilde{\dot{y}}, \\
    s^k(\alpha) = & \ s^k - \alpha \tilde{\dot{s}},
  \end{align}
\end{subequations}
where $(\tilde{\dot{x}}, \tilde{\dot{y}}, \tilde{\dot{s}})$ are obtained by \eqref{resolution_first_derivative_from_MNES}.

For the purpose of proving convergence,
we introduce the following expressions in a form similar to the previously defined in \eqref{def_G_g_h}:
\begin{subequations}
  \label{def_G_g_h_line}
  \begin{align}
    \hat{G}_i^k(\alpha) & = x_i^k(\alpha) s_i^k(\alpha) - \gamma_1 \mu_k(\alpha) \text { for } i \in\{1, \ldots, n\}, \\
    \hat{g}^k(\alpha)   & = x^k(\alpha)^\top s^k(\alpha) - (1 - \alpha)(x^k)^\top s^k,                                \\
    \hat{h}^k(\alpha)   & = \left(1 - (1-\beta) \alpha\right)(x^k)^\top s^k-x^k(\alpha)^\top s^k(\alpha),
  \end{align}
\end{subequations}
where $x^k(\alpha)$ and $s^k(\alpha)$ are defined by \eqref{def_variable_alpha_with_inexact_derivatives_line}.
As in the proposed method, the next iterate is computed using a step size $\alpha$ that satisfies the following conditions:
\begin{equation}
  \label{conditions_G_g_h_no_less_than_0_line}
  \hat{G}_i^k(\alpha) \ge 0 \text { for } i \in\{1, \ldots, n\}, \quad
  \hat{g}^k(\alpha) \ge 0, \quad
  \hat{h}^k(\alpha) \ge 0.
\end{equation}
Under this condition, the next iteration point $(x^k(\alpha), y^k(\alpha), s^k(\alpha))$ is in the neighborhood $\Neighborhood(\gamma_1, \gamma_2)$~\cite[Lemma~4.5]{mohammadisiahroudi2024efficient}.

Algorithm~\ref{algorithm_II_line_IPM} gives the framework of the II-line.
\begin{algorithm}[ht]
  \caption{The inexact infeasible line-search interior-point method (II-line)}
  \label{algorithm_II_line_IPM}
  \begin{algorithmic}[1]
    \Require $\zeta > 0$,
    $\gamma_1 \in (0,1)$,
    $\gamma_2 \ge 1$,
    $\sigma, \eta, \beta$ satisfying \eqref{parameter_conditions} and
    an initial point $(x^0, y^0, s^0)$ meeting \eqref{def_initial_point}.
    \Ensure $\zeta$-optimal solution $(x^k, y^k, s^k)$
    \State $k \leftarrow 0$
    \While {$(x^k, y^k, s^k) \notin S_\zeta$}
    \State $\mu_k \leftarrow (x^k)^\top s^k / n$
    \State Calculate $(\tilde{\dot{x}}, \tilde{\dot{y}}, \tilde{\dot{s}})$ by solving \eqref{first_derivative_MNES} inexactly satisfying \eqref{def_upper_derivatives_residual_MNES}.
    \State $\alpha_k \leftarrow \max \left\{\alpha \in (0, 1] \mid \alpha \text{ satisfies } \eqref{conditions_G_g_h_no_less_than_0_line}\right\}$.
    \label{line_algo_II_line_search_decide_step_size}
    \State Set $(x^{k+1}, y^{k+1}, s^{k+1}) = (x^k(\alpha_k), y^k(\alpha_k), s^k(\alpha_k))$ by \eqref{def_variable_alpha_with_inexact_derivatives_line}.
    \State $k \leftarrow k+1$
    \EndWhile
  \end{algorithmic}
\end{algorithm}

The theoretical iteration complexity of Algorithm~\ref{algorithm_II_line_IPM}, which is $\order{n^2 L}$,
is proven in the appendix~\ref{section_proof_for_inexact_line}.

\subsection{Test problems}
\label{section_test_problems}
Iterative solvers are often employed when the matrix related to the normal equation is very large and the Cholesky factorization is impractical.
In this context,
we use the largest problems
in the NETLIB collection~\cite{browne1995netlib};
DFL001, QAP15, STOCFOR3 and the fifteen Kennington problems~\cite{carolan1990empirical}.
We used the same preprocessing as \citet[Section~5.1]{iida2024infeasible},
e.g., removing redundant rows of the constraint matrix $A$.

Table~\ref{table_problem_information} shows the information about the problems.
The first column of the table is the problem name,
and the second and the third are the variable size $n$ and the number of constraints $m$, respectively, after preprocessing denoted as above.
The fourth column is the number of nonzero elements in the constraint matrix $A$,
and the last column shows the density of $A$.

\begin{table}[htb]
  \centering
  \caption{Problem information}
  \label{table_problem_information}
  \begin{tabular}{lrrrr}
    \hline
    problem  & $n$    & $m$    & nonzeros & density($(\text{nonzeros} / m n) \times 10^4$) \\
    \hline \hline
    CRE-A    & 6997   & 3299   & 17559    & 7.61                                           \\
    CRE-B    & 36382  & 5336   & 112249   & 5.78                                           \\
    CRE-C    & 5684   & 2647   & 14139    & 9.40                                           \\
    CRE-D    & 28601  & 4102   & 86734    & 7.39                                           \\
    DFL001   & 11853  & 5713   & 35449    & 5.23                                           \\
    KEN-07   & 5127   & 3951   & 11121    & 5.49                                           \\
    KEN-11   & 32996  & 26341  & 71032    & 0.82                                           \\
    KEN-13   & 72784  & 58757  & 155144   & 0.36                                           \\
    KEN-18   & 255248 & 205676 & 553134   & 0.11                                           \\
    OSA-07   & 25067  & 1118   & 144812   & 51.67                                          \\
    OSA-14   & 54797  & 2337   & 317097   & 24.76                                          \\
    OSA-30   & 104374 & 4350   & 604488   & 13.31                                          \\
    OSA-60   & 243246 & 10280  & 1408073  & 5.63                                           \\
    PDS-06   & 36920  & 17604  & 78016    & 1.20                                           \\
    PDS-10   & 63905  & 30773  & 135020   & 0.69                                           \\
    PDS-20   & 139330 & 65437  & 293841   & 0.32                                           \\
    QAP15    & 22275  & 6330   & 94950    & 6.73                                           \\
    STOCFOR3 & 21910  & 15044  & 89746    & 2.72                                           \\
    \hline
  \end{tabular}
\end{table}

\subsection{Implementation details}
\label{section_implementation_details}
We implemented the proposed method (II-arc) based on Algorithm~\ref{algorithm_II_arc_IPM}
and the inexact infeasible line-search IPM (II-line) based on Algorithm~\ref{algorithm_II_line_IPM}.
In this section,
we describe the implementation details before discussing the results.

\subsubsection{Parameter settings}
In these numerical experiments,
we set
$$
  \sigma = 0.4, \
  \eta = 0.3, \
  \gamma_1 = 0.1, \
  \gamma_2 = 1, \
  \beta = 0.9.
$$

These parameters are chosen to satisfy \eqref{parameter_conditions}.

\subsubsection{Solving systems of linear equations}
\label{section_solving_systems_of_linear_equations}
To solve the systems of linear equations inexactly,
we employ preconditioned CG in Scipy.
There are many other iterative solvers for solving systems,
such as the minimum residual (MINRES), and the generalized minimum residual (GMRES) methods.
The preliminary experiments showed that CG was the fastest inexact solver  in Scipy for II-arc and II-line.

Although the proposed method adopts the MNES formulation in Section~\ref{section_proposed_method} for theoretical analysis,
preliminary experiments revealed that MNES suffers from numerical instability in practice.
In theory,
the result of \citet{monteiro2004uniform} implies that MNES avoids the spectral deterioration of unpreconditioned normal equations.
Despite this theoretical guarantee,
the practical performance of iterative solvers such as CG on MNES may still be unsatisfactory in all test problems.
In our implementation,
the basis $\hat{B}$ required for constructing MNES is obtained using the maximum weight basis algorithm~\cite{monteiro2003convergence},
which relies on repeated rank computations to enforce linear independence.
In finite precision arithmetic, this procedure may still result in a singular submatrix $A_{\hat{B}}$,
or the repeated rank evaluations incur a substantial computational overhead for large-scale instances
(see Section~\ref{section_stopping_criteria} for a detailed time-limit).

Therefore,
we rely on empirical evidence to assess the numerical stability of MNES in our experimental setting,
and, based on these observations,
we use the NES formulations \eqref{first_derivative_NES} and
\begin{equation}
  M^k \ddot{y} = \rho_2^k, \label{second_derivative_NES}
\end{equation}
instead of the MNES \eqref{first_derivative_MNES} and \eqref{second_derivative_MNES},
respectively in the numerical experiments.
The inexact solution of \eqref{first_derivative_NES} satisfies \eqref{inexact_first_derivative_NES},
and that of \eqref{second_derivative_NES} satisfies
\begin{equation}
  M^k \tilde{\ddot{y}} = \rho_2^k + r_2^k,
  \label{inexact_second_derivave_NES}
\end{equation}
where the error $r_2^k$ is defined similar to $r_1^k$.
As for the solution accuracy,
we set the following threshold as in \eqref{def_upper_derivatives_residual_MNES}:
\begin{equation}
  \norm{r_i^k} \leq \eta \frac{\sqrt{\mu_k}}{\sqrt{n}}
  \quad \forall i \in \{1, 2\}.
  \label{def_upper_derivatives_residual_NES}
\end{equation}

In the numerical experiments,
we solve the standard NES linear systems \eqref{first_derivative_NES} and \eqref{second_derivative_NES} using preconditioned CG with a Jacobi (diagonal) preconditioner.
Here ``preconditioned CG'' refers to preconditioning of the iterative solver,
whereas the ``preconditioned NES'' formulation analyzed by \citet{monteiro2004uniform} corresponds to our MNES formulation used in Section~\ref{section_proposed_method} for the complexity analysis.
The reason why we applied a Jacobi preconditioner is that it is simpler than the others,
such as the controlled Cholesky Factorization preconditioner~\cite{bocanegra2007using},
the splitting preconditioner~\cite{oliveira2005new},
the hybrid of these~\cite{bartmeyer2021switching},
the Schur complement approach~\cite{scott2018schur},
and the method by Bergamaschi et al.~\cite{bergamaschi2021new}.
As mentioned above,
the main purpose of the experiments in this paper is to
compare the proposed method (II-arc) and the existing method
(II-line), and
the selection of the preconditioner with full numerical experiments is beyond the scope of this paper and can be considered separately.

II-arc needs to solve two systems---\eqref{inexact_first_derivative_NES} and \eqref{inexact_second_derivave_NES}---at each iteration,
while II-line needs to solve one---\eqref{inexact_first_derivative_NES}.
Therefore,
the reduction in the iteration complexity discussed in Section~\ref{section_proposed_method}
does not ensure that in the computational cost.

\subsubsection{The modification of the second order derivatives}
\label{section_modification_second_order_derivative}
When $\norm{-2 \tilde{\dot{x}} \circ \tilde{\dot{s}}}_\infty \le \eta \mu_k$ is satisfied,
\eqref{indexact_second_derivative_conditions} and \eqref{upper_residual_term_MNES} can hold with
$(\tilde{\ddot{x}}, \tilde{\ddot{y}}, \tilde{\ddot{s}}) = (0,0,0)$.
Therefore,
to shorten the computation time,
we skip solving \eqref{second_derivative_NES} by the preconditioned CG and set $(\tilde{\ddot{x}}, \tilde{\ddot{y}}, \tilde{\ddot{s}}) = (0,0,0)$.
In this case,
\eqref{def_variable_alpha_with_inexact_derivatives} can be interpreted as a line-search method.

Furthermore,
if the inexact solution $\tilde{\ddot{y}}$ of \eqref{second_derivative_NES}
satisfies $\norm{M_2^k \tilde{\ddot{y}}} - \rho_2^k > \norm{M^k 0 - \rho_2^k} = \norm{-\rho_2^k}$
then the zero vector is adopted as the solution of the linear system.
Therefore,
we replace $\tilde{\ddot{y}}$ with a zero vector to avoid a large error.
This heuristic aligns with the approach outlined by Mizuno and Jarre~\cite{mizuno1999global}.

\subsubsection{Step size}
At line~\ref{line_algo_II_arc_search_decide_step_size} of Algorithm~\ref{algorithm_II_arc_IPM} and line~\ref{line_algo_II_line_search_decide_step_size} of Algorithm~\ref{algorithm_II_line_IPM},
Armijo's rule~\cite{wright1997primal} is employed to determine an actual step size $\alpha_k$,
since \eqref{conditions_G_g_h_no_less_than_0}
does not give an analytical solution.

\subsubsection{Stopping criteria}
\label{section_stopping_criteria}

The algorithms terminate when the iterate satisfies the stopping criterion,
i.e. $(x^k, y^k, s^k) \in \SolSet_\zeta$.
In this paper, we set $\zeta = 10^{-4}$.
Heuristics can be implemented to achieve higher accuracy;
however, since the main purpose of the numerical experiments here is to compare II-arc and II-line,
discussion of heuristics is left for future work.

In addition,
we stop the algorithms prematurely when
the step size $\alpha_k$ diminishes as $\alpha_k < 10^{-7}$
or the algorithms exceed the time limit of 72,000 seconds.

\subsubsection{Initial points}
Many methods have been proposed for the initial point.
For example,
the method of Yang~\cite[Section~4.1]{Yang2017} generates candidates using the Mehrotra method~\cite{Mehrotra1992} and the Lustig method~\cite{lustig1992implementing}, and select the better one as the initial point.
Although heuristic initial points such as those typically enter the neighborhood $\Neighborhood(\gamma_1, \gamma_2)$ after the first iteration in practice,
the manner in which the initial point is constructed can still have a non-negligible impact on numerical behavior,
especially for large-scale problems.
In particular,
the initialization procedure of \citet{Yang2017} requires solving additional linear systems,
such as $A A^\top y = A c$,
in order to generate the candidate initial points.
For large-scale instances, directly applying Cholesky factorization to such systems can be computationally expensive,
and therefore iterative linear solvers must be employed.
In this case,
the convergence behavior, numerical stability, and overall computational cost may depend on the properties of these auxiliary linear systems and on the performance of the iterative solver.

On the other hand,
the constant initial point defined in~\eqref{def_initial_point} is constructed explicitly to satisfy the theoretical assumptions $(x^0, y^0, s^0) \in \Neighborhood(\gamma_1, \gamma_2)$
and does not require solving any additional linear systems.

Based on these observations,
we compare the initial points proposed by \citet{Yang2017} with ones defined in \eqref{def_initial_point},
considering not only the number of iterations and computation time but also numerical stability and robustness in large-scale settings.

Hereafter,
we denote the choice of the initial point by appending a suffix to the name of each algorithm.
For example,
II-arc initialized by the method of \citet{Yang2017} is referred to as II-arc-Yang,
whereas II-arc initialized by \eqref{def_initial_point} is referred to as II-arc-Constant.
A detailed comparison between II-arc-Yang and II-arc-Constant is presented in Section~\ref{section_numerical_results_initial_point}.

We describe the specific settings for each initial point strategy as follows.
In the method of \citet{Yang2017},
linear systems for obtaining candidate initial points are solved using the same preconditioned CG as described in Section~\ref{section_solving_systems_of_linear_equations}.
The stopping criteria are set to the default values in SciPy~(e.g., $\norm{b - Ax} \le \epsilon \norm{b}$, where $\epsilon = 10^{-5}$).
When the initial point $(x^0, y^0, s^0)$ is chosen according to \eqref{def_initial_point},
we select a sufficiently large $\omega=10^4$.

\subsection{Numerical Results}
\label{section_numerical_results}
First, we compare II-arc-Yang with II-arc-Constant by solving the benchmark problems,
and show that II-arc-Yang achieves better performance while maintaining numerical stability.
Second, we compare II-arc-Yang with II-line-Yang
and show that II-arc-Yang can solve the large problems with fewer iterations and less computation time.

Table~\ref{table_results_for_comparison} reports the detailed numerical results.
The first column of the table is the problem name,
and second to last columns report the number of iterations and the computation time (in seconds).
The underlined results indicate the best results among the four methods.
A mark `-' indicates
the algorithms stop before reaching the optimality described in Section~\ref{section_stopping_criteria}.

\begin{table}[htb]
  \centering
  \caption{Numerical results on the proposed method and the existing methods}
  \label{table_results_for_comparison}
  \begin{tabular}{l|rr|rr|rr}
    \hline
    problem                              &
    \multicolumn{2}{c|}{II-arc-Constant} &
    \multicolumn{2}{c|}{II-arc-Yang}     &
    \multicolumn{2}{c}{II-line-Yang}                                                                                                              \\
                                         & Itr.          & Time                & Itr.          & Time                & Itr. & Time                \\
    \hline \hline
    CRE-A                                & 46            & 44.15               & \underbar{42} & \underbar{41.22}    & 127  & 49.01               \\
    CRE-B                                & 75            & \underbar{306.5}    & \underbar{68} & 307.9               & 258  & 399.5               \\
    CRE-C                                & \underbar{49} & \underbar{54.82}    & \underbar{49} & 56.04               & 125  & 58.23               \\
    CRE-D                                & 75            & 185.69              & \underbar{70} & \underbar{177.53}   & 244  & 246.13              \\
    DFL001                               & 90            & 3019.68             & \underbar{67} & 2826.97             & 123  & \underbar{2187.42}  \\
    KEN-07                               & \underbar{33} & 23.91               & \underbar{33} & 23.44               & 39   & \underbar{22.29}    \\
    KEN-11                               & 42            & \underbar{2472.35}  & \underbar{41} & 2545.39             & 60   & 2597.03             \\
    KEN-13                               & \underbar{51} & 18246.07            & 52            & 17879.89            & 82   & \underbar{16343.66} \\
    KEN-18                               & -             & -                   & -             & -                   & -    & -                   \\
    OSA-07                               & 45            & 4.78                & \underbar{35} & \underbar{4.25}     & 51   & 5.04                \\
    OSA-14                               & 50            & 10.92               & \underbar{38} & \underbar{9.8}      & 59   & 12.97               \\
    OSA-30                               & 49            & 19.66               & \underbar{39} & \underbar{17.67}    & 61   & 24.09               \\
    OSA-60                               & 52            & \underbar{54.17}    & \underbar{46} & 55.2                & -    & -                   \\
    PDS-06                               & 61            & \underbar{127.63}   & \underbar{60} & 131.49              & 97   & 129.61              \\
    PDS-10                               & 76            & 456.39              & \underbar{71} & \underbar{418.36}   & 132  & 456.93              \\
    PDS-20                               & \underbar{94} & 18629.42            & 95            & \underbar{18407.83} & 191  & 25110.81            \\
    QAP15                                & 28            & 38.76               & \underbar{14} & \underbar{34.93}    & 23   & 35.73               \\
    STOCFOR3                             & 50            & \underbar{32249.56} & \underbar{36} & 33710.26            & 71   & 37332.84            \\
    \hline
  \end{tabular}
\end{table}

In this section,
we use a performance profile~\cite{dolan2002benchmarking,gould2016note} for comparing the methods.
In a performance profile,
the horizontal axis represents the scaling parameter $\tau$,
while the vertical axis $P(r \le \tau)$ denotes the fraction of test problems.
When the number of iterations is considered,
$P(r \le \tau)$ indicates the proportion of problems for which the algorithm requires no more than $\tau$ times the number of iterations of the best-performing method.
Hence, an algorithm is regarded as more efficient if it achieves higher values of $P(r \le \tau)$ for smaller $\tau$.
The figures on the performance profile in this section were generated with a Julia package~\cite{orban-benchmarkprofiles-2019}.

\subsubsection{Comparison of initial-point strategies}
\label{section_numerical_results_initial_point}
Both II-arc-Yang and II-arc-Constant reached the stopping criterion $(x^k, y^k, s^k) \in \SolSet_\zeta$
in 17 out of 18 test problems, failing only on KEN-18.
Therefore,
the overall results indicate that the proposed arc-search methods exhibit satisfactory numerical stability and are insensitive to the choice of initial points.

Figure~\ref{fig_comparison_iter_num_based_on_inital_point} presents a performance profile of the number of iterations for II-arc-Yang and II-arc-Constant.
For these problems,
II-arc-Yang consistently requires fewer iterations than II-arc-Constant,
demonstrating that the heuristic of \citet{Yang2017} effectively reduces the computational effort.

\begin{figure}[htb]
  \centering
  \includegraphics[width=0.7\textwidth]{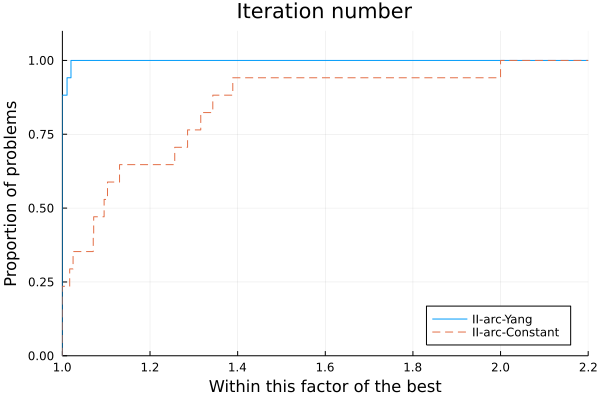}
  \caption{Performance profile of the number of iterations with II-arc-Yang and II-arc-Constant}
  \label{fig_comparison_iter_num_based_on_inital_point}
\end{figure}

Figure~\ref{fig_comparison_calc_time_based_on_inital_point} shows a performance profile of the computation time for II-arc-Yang and II-arc-Constant.
Even when accounting for the additional cost of solving linear systems to compute the initial point,
II-arc-Yang remains superior in terms of overall computation time.

\begin{figure}[htb]
  \centering
  \includegraphics[width=0.7\textwidth]{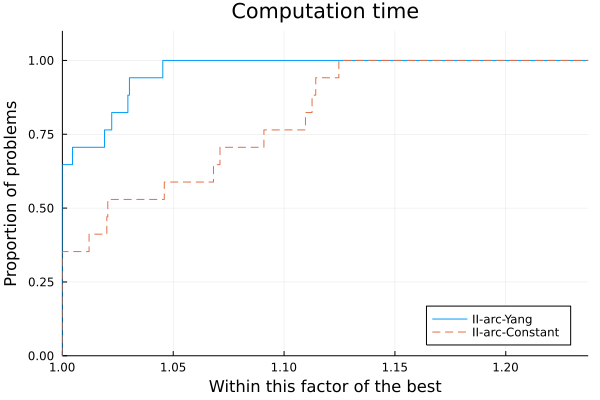}
  \caption{Performance profile of the computation time with II-arc-Yang and II-arc-Constant}
  \label{fig_comparison_calc_time_based_on_inital_point}
\end{figure}

Based on these observations, we adopt the initial point by \citet{Yang2017} in the remainder of this paper and conduct a comparison with the existing II-line-Yang method.

\subsubsection{Comparison with existing method}
As mentioned above, II-arc-Yang successfully solved 17 out of the 18 test problems,
whereas II-line-Yang solved 16 problems, failing on KEN-18 and OSA-60.

Firstly,
Figure~\ref{fig_comparison_iter_num_with_existing_method} shows a performance profile on the numbers of iterations of II-arc-Yang and II-line-Yang,
filtered by the problems solved by all methods.
We observe from Figure~\ref{fig_comparison_iter_num_with_existing_method} that
II-arc-Yang used fewer iterations than II-line-Yang in all the problems.
For nearly 25\% of the test problems,
II-line-Yang required twice or more iterations than II-arc-Yang.
Therefore,
these results indicate that the number of iterations can be reduced by approximating the central path with the ellipsoidal arc,
when the search directions are obtained inexactly.

\begin{figure}[htb]
  \centering
  \includegraphics[width=0.7\textwidth]{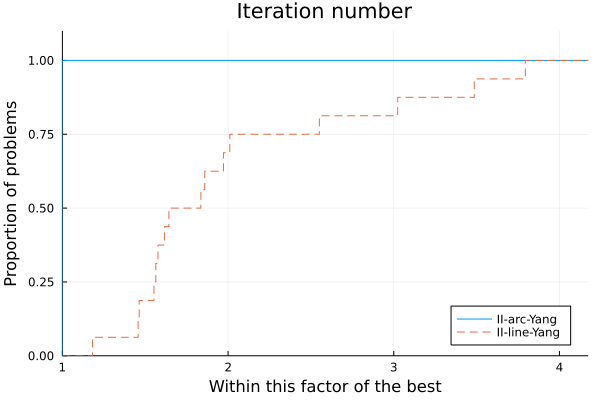}
  \caption{Performance profile of the number of iterations with II-arc-Yang and II-line-Yang}
  \label{fig_comparison_iter_num_with_existing_method}
\end{figure}

Figure~\ref{fig_comparison_calc_time_with_existing_method} provides a performance profile on the computation time.
We can observe from this figure that
II-arc-Yang used less computation time than II-line-Yang in 75\% of the problems.
Moreover,
the performance profiles indicate that II-arc-Yang attains a larger performance ratio for smaller values of the performance factor,
meaning that it solves a higher proportion of problems within a smaller factor of the best computation time than II-line-Yang.
This suggests that the reduction in the number of the iterations contributes to the overall computational efficiency of II-arc-Yang.

\begin{figure}[htb]
  \centering
  \includegraphics[width=0.7\textwidth]{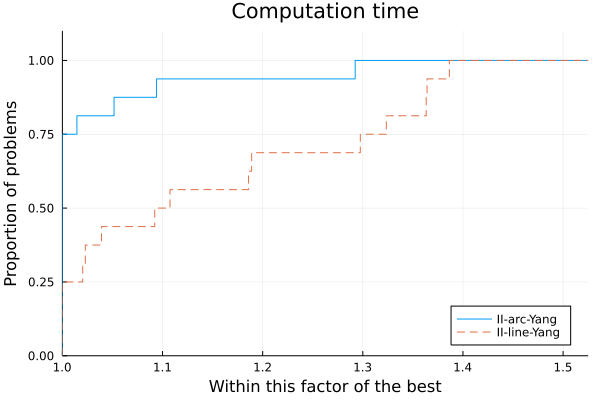}
  \caption{Performance profile of the computation time with II-arc-Yang and II-line-Yang}
  \label{fig_comparison_calc_time_with_existing_method}
\end{figure}

II-arc-Yang requires solving an additional system of linear equations~\eqref{inexact_second_derivave_NES} compared to II-line-Yang.
However,
doubling the number of systems at each iteration does not simply mean doubling the total computation time.
Figure~\ref{fig_comparison_calc_time_to_obtain_search_directions} shows the computation time to obtain the search directions
(the execution time of the preconditioned CG to solve \eqref{inexact_first_derivative_NES} and \eqref{inexact_second_derivave_NES} in II-arc-Yang and \eqref{inexact_first_derivative_NES} in II-line-Yang)
when solving PDS-10.
This figure shows that the computation time of the preconditioned CG increases in both II-arc-Yang and II-line-Yang as the iterations proceed,
since $\mu_k$ on the right-hand side in \eqref{def_upper_derivatives_residual_NES} is smaller in the later iterations.
The computation time of II-arc-Yang was reduced to a half after the 65th iteration,
since the calculation on $(\tilde{\ddot{x}}, \tilde{\ddot{y}}, \tilde{\ddot{s}})$ was skipped as described in Section~\ref{section_modification_second_order_derivative}.

\begin{figure}[htb]
  \centering
  \includegraphics[width=0.7\textwidth]{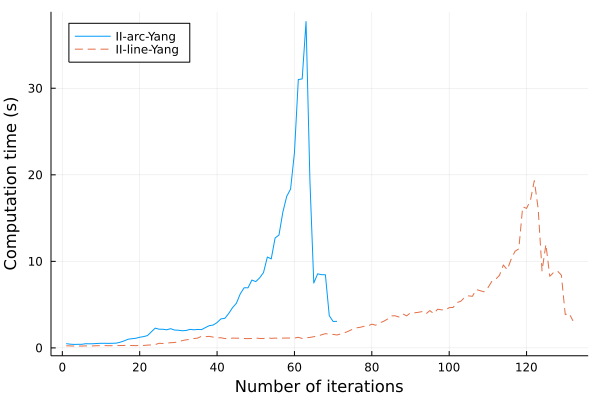}
  \caption{Computation time to obtain the search directions for solving PDS-10}
  \label{fig_comparison_calc_time_to_obtain_search_directions}
\end{figure}
\section{Conclusion}
\label{section_conclusion}
In this work,
we proposed an inexact infeasible arc-search interior-point method (II-arc) for solving LOPs.
In particular,
we showed that the proposed method achieves a smaller polynomial iteration complexity than II-line by a factor of $n^{0.5}$.
In the numerical experiments,
II-arc required fewer iterations than II-line,
since the ellipsoidal arc in II-arc can approximate the central path more adequately than the straight line used in II-line.
In addition,
this iteration reduction enabled II-arc-Constant to reduce the computational time.

Since the primary objective of this paper was to propose the inexact infeasible arc-search IPM and to discuss its convergence,
the numerical performance of the proposed method can be further improved.
For example, the method can be implemented with other programming languages (such as C++ and Julia).
In addition, inexact linear system solvers, including \cite{bartmeyer2021switching},
and their preconditioner should be investigated.
The sensitivity of the inexact linear system solvers can also be discussed, in particular,
in the case of ill-conditioned input matrix $A$.
From the viewpoint of the reduction in the number of iterations,
the combination with Nesterov's restarting strategy~\cite{iida2024infeasible} or
feasible IPMs~\cite{mohammadisiahroudi2025inexact,mohammadisiahroudi2018improvements} can be a future topic.

\appendix
\section{Appendix}

\subsection{Theoretical proof of the iteration complexity of II-line}
\label{section_proof_for_inexact_line}
In this section,
we present the theoretical proof of the iteration complexity of the II-line Algorithm~\ref{algorithm_II_line_IPM} described in Section~\ref{section_II_line}.
All notation and assumptions follow those in the main text.

The proof is carried out in a manner analogous to the analysis of the II-arc method presented in Section~\ref{section_theoretical_proof}.
Therefore, it is sufficient to establish the following proposition in order to conclude that the iteration complexity is $\order{n^2 L}$.
\begin{proposition}
  \label{proposition_lower_bound_of_step_size_line}
  Let $\{(x^k, y^k, s^k)\}$ be the sequence generated by Algorithm~\ref{algorithm_II_line_IPM}.
  Then,
  there exists $\hat{\alpha} > 0$ satisfying
  \eqref{conditions_G_g_h_no_less_than_0} for any $\alpha_k \in (0, \hat{\alpha}]$ and
  $$
    \sin(\hat{\alpha}) = \frac{C}{n^2},
  $$
  where $C$ is a positive constant.
\end{proposition}

Lemma~\ref{lemma_inexact_solution_MNES_conditions} also applies to II-line,
since its expressions are identical to those used in II-arc.

By the difference of the step size,
Lemma~\ref{lemma_decrease_constraint_residuals} changes for II-line as follows:
\begin{lemma}[{\cite[Lemma~7.2]{yang2020arc}}]
  \label{lemma_decrease_constraint_residuals_line}
  For each iteration $k$,
  the following relations hold.
  \begin{align*}
    r_b(x^{k+1})          & = r_b(x^k)\left(1 - \alpha_k \right),       \\
    r_c(y^{k+1}, s^{k+1}) & = r_c(y^k, s^k) \left(1 - \alpha_k \right).
  \end{align*}
\end{lemma}
This Lemma holds by Lemma~\ref{lemma_inexact_solution_MNES_conditions} and \eqref{def_variable_alpha_with_inexact_derivatives_line}.
For the following discussions,
we introduce the notation:
\begin{equation}
  \nu_k = \prod_{i=0}^{k-1} (1 - \alpha_i).
  \label{def_nu_line}
\end{equation}

Lemmas~\ref{lemma_upper_nu_x_s}, \ref{lemma_upper_derivatives_residual} and \ref{lemma_first_derivative_upper} also hold under the new definition of $\nu_k$,
in the same manner as in the proposed method.
Using these lemmas, we are ready to prove Proposition~\ref{proposition_lower_bound_of_step_size_line}.

\begin{proof}[Proof of Proposition~\ref{proposition_lower_bound_of_step_size_line}]
  We can obtain
  \begin{align}
    x^k(\alpha) \circ s^k(\alpha)
    = & \ \left(x^k-\alpha\tilde{\dot{x}}\right) \circ \left(s^k-\alpha\tilde{\dot{s}}\right) \notag                                                        \\
    = & \ x^k \circ s^k - \alpha \left(x^k \circ \tilde{\dot{s}} + \tilde{\dot{x}} \circ s^k\right) + \alpha^2 \tilde{\dot{x}} \circ \tilde{\dot{s}} \notag \\
    [\because \eqref{inexact_first_derivative_MNES_duality}] \quad =
      & \ x^k \circ s^k - \alpha (x^k \circ s^k - \sigma \mu_k e - S^k v_1^k) + \alpha^2 \tilde{\dot{x}} \circ \tilde{\dot{s}} \notag                       \\
    = & \ x^k \circ s^k (1 - \alpha) + \alpha \sigma \mu_k e + \alpha S^k v_1^k + \alpha^2 \left(\tilde{\dot{x}} \circ \tilde{\dot{s}}\right)
    \label{x_s_alpha_Hadamard_line}
  \end{align}
  and
  \begin{align}
    x^k(\alpha)^\top s^k(\alpha)
    = & \ \left(x^k- \alpha \tilde{\dot{x}}\right)^\top \left(s^k- \alpha \tilde{\dot{s}}\right) \notag \\
    [\because \eqref{x_s_alpha_Hadamard_line}, \eqref{def_mu}] \quad
    = & \ (x^k)^\top s^k \left((1 - \alpha) + \alpha \sigma \right)
    + \alpha^2 \tilde{\dot{x}}^\top \tilde{\dot{s}} + \alpha \sum_{i=1}^n [S^k v_1^k]_i.
    \label{x_s_alpha_inner_product_line}
  \end{align}

  From Lemma~\ref{lemma_first_derivative_upper} and the Cauchy-Schwartz inequality,
  we obtain \eqref{upper_product_of_dot_x_and_dot_s_element_wise}.

  We prove that the step size $\alpha$ satisfying $\hat{g}^k(\alpha) \ge 0$ is bounded away from zero.
  From \eqref{x_s_alpha_inner_product_line},
  \begin{align}
    x^k(\alpha)^\top s^k(\alpha)
    \ge & \ (x^k)^\top s^k \left((1 - \alpha) + \alpha \sigma \right)
    - \alpha^2 \abs{\tilde{\dot{x}}^\top \tilde{\dot{s}}} - \alpha \norm{S^k v_1^k}_1 \notag \\
    [\because \eqref{inequality_norms}, \eqref{upper_residual_term_MNES}] \quad
    \ge & \ (x^k)^\top s^k \left((1 - \alpha) + \alpha \sigma \right)
    - \alpha^2 \abs{\tilde{\dot{x}}^\top \tilde{\dot{s}}} - \alpha \eta n \mu_k.
    \label{lower_x_s_alpha_inner_product_line}
  \end{align}
  Therefore,
  \begin{align*}
    \hat{g}^k(\alpha) = & \ x^k(\alpha)^\top s^k(\alpha) - (1 - \alpha)(x^k)^\top s^k                                                \\
    [\because \eqref{lower_x_s_alpha_inner_product_line}] \quad
    \ge                 & \ \alpha \sigma (x^k)^\top s^k - \alpha \eta n \mu_k - \alpha^2 \abs{\tilde{\dot{x}}^\top \tilde{\dot{s}}} \\
    [\because \eqref{def_mu}, \eqref{upper_product_of_dot_x_and_dot_s_element_wise}] \quad
    \ge                 & \ \alpha n \mu_k \left((\sigma - \eta) - 2 \alpha C_2^2 n \right).
  \end{align*}
  Since $- \alpha C_2^2 n$ is monotonically decreasing by $\alpha$
  and $\sigma > \eta$ holds from \eqref{parameter_condition_for_G_i} and $\gamma_1 \in (0, 1)$,
  there exists the step size $\hat{\alpha}_1 \in (0, 1]$ satisfying the last formula of the right-hand side is no less than 0.
  When
  $$
    \hat{\alpha}_1 \le \frac{\sigma - \eta}{2 n C_2^2},
  $$
  from $0 < \sigma -\eta < \sigma \le 1$,
  \begin{equation*}
    (\sigma - \eta) - 2 C_2^2 n \hat{\alpha}_1 \ge 0.
  \end{equation*}
  Therefore,
  $\hat{g}^k(\alpha) \ge 0$ is satisfied for any $\alpha \in (0, \hat{\alpha}_1]$.

  Next, we consider the range of $\alpha$ such that $\hat{G}^k_i(\alpha) \ge 0$.
  From \eqref{upper_product_of_dot_x_and_dot_s_element_wise} and \eqref{upper_derivatives_element_wise_minus_products_dot},
  we have
  \begin{align*}
    \hat{G}_i^k(\alpha) = & \ x_i^k(\alpha) s_i^k(\alpha) - \gamma_1 \mu_k(\alpha)                                                                                               \\
    [\because \eqref{x_s_alpha_Hadamard_line}, \eqref{def_mu}, \eqref{x_s_alpha_inner_product_line}] \quad
    \ge                   & \ x^k_i s^k_i (1 - \alpha) + \alpha \sigma \mu_k
    + \alpha^2 \tilde{\dot{x}}_i \tilde{\dot{s}}_i - \alpha \norm{S^k v_1^k}_\infty                                                                                              \\
                          & - \frac{\gamma_1}{n} \left( n \mu_k (1 - \alpha + \alpha \sigma) + \alpha^2 \tilde{\dot{x}}^\top \tilde{\dot{s}} + \alpha \norm{S^k v_1^k}_1 \right) \\
    [\because \eqref{def_neighborhood}, \eqref{upper_residual_term_MNES}, \eqref{inequality_norms}] \quad
    \ge                   & \ \alpha (1 - \gamma_1) \sigma \mu_k - \alpha (1 + \gamma_1) \eta \mu_k
    + \alpha^2 \left(\tilde{\dot{x}}_i \tilde{\dot{s}}_i - \frac{\gamma_1}{n} \tilde{\dot{x}}^\top \tilde{\dot{s}}\right)                                                        \\
    [\because \eqref{upper_product_of_ddot_x_and_ddot_s_element_wise}, \eqref{upper_derivatives_element_wise_minus_products_dot}] \quad
    \ge                   & \ \alpha \mu_k \left( (1 - \gamma_1) \sigma - (1 + \gamma_1) \eta - 2 \alpha C_2^2 n^2 \right).
  \end{align*}
  We can derive the same discussion as $\hat{g}^k(\alpha)$ using \eqref{parameter_condition_for_G_i}.
  When
  $$
    \hat{\alpha}_2 \le \frac{(1 - \gamma_1) \sigma - (1 + \gamma_1) \eta}{2 n^2 C_2^2},
  $$
  from $0 < (1 - \gamma_1) \sigma - (1 + \gamma_1) \eta < \sigma \le 1$,
  \begin{equation*}
    (1 - \gamma_1) \sigma - (1 + \gamma_1) \eta - 2 \alpha C_2^2 n^2 \ge 0.
  \end{equation*}
  Therefore,
  $\hat{G}_i^k(\alpha) \ge 0$ is satisfied for $\alpha \in (0, \hat{\alpha}_2]$.

  Lastly, we consider $\hat{h}^k(\alpha) \ge 0$.
  Similarly to the derivation of \eqref{lower_x_s_alpha_inner_product_line},
  we can obtain the following:
  \begin{equation}
    x^k(\alpha)^\top s^k(\alpha)
    \le (x^k)^\top s^k \left((1 - \alpha) + \alpha \sigma \right) + \alpha^2 \abs{\tilde{\dot{x}}^\top \tilde{\dot{s}}} + \alpha \eta n \mu_k,
    \label{upper_x_s_alpha_inner_product_line}
  \end{equation}
  Therefore,
  \begin{align*}
    \hat{h}^k(\alpha) = & \ \left(1 - \alpha (1-\beta) \right)(x^k)^\top s^k - x^k(\alpha)^\top s^k(\alpha) \\
    [\because \eqref{upper_x_s_alpha_inner_product_line}] \quad
    \ge                 & \ (x^k)^\top s^k \left(\alpha \beta - \alpha \sigma \right)
    - \alpha^2 \abs{\tilde{\dot{x}}^\top \tilde{\dot{s}}} - \alpha \eta n \mu_k                             \\
    [\because \eqref{def_mu}, \eqref{upper_product_of_dot_x_and_dot_s_element_wise}] \quad
    \ge                 & \ \alpha n \mu_k \left((\beta - \sigma - \eta) - \alpha C_2^2 n\right)
  \end{align*}
  $\left((\beta - \sigma - \eta) - \alpha C_2^2 n\right)$ is monotonically decreasing for $\alpha$.
  Therefore,
  it is possible to take a step size $\hat{\alpha}_3$ satisfying $\hat{h}^k(\hat{\alpha}_3) \ge 0$ from \eqref{parameter_condition_beta_more_than_sigma_plus_eta}.
  When
  $$
    \hat{\alpha}_3 \le \frac{\beta - \sigma - \eta}{n C_2^2 },
  $$
  from $0 < \beta - \sigma -\eta < \beta < 1$, we know
  \begin{equation*}
    (\beta - \sigma - \eta) - \hat{\alpha}_3 C_2^2 n \ge 0.
  \end{equation*}
  Therefore, $\hat{h}^k(\alpha) \ge 0$ is satisfied for $\alpha \in (0, \hat{\alpha}_3]$.

  From the above discussions,
  when $\hat{\alpha}$ is taken such that
  \begin{equation*}
    \sin(\hat{\alpha}) = \frac{1}{n^2} \frac{
      \min \left\{(1 - \gamma_1) \sigma - (1 + \gamma_1) \eta, \beta - \sigma - \eta\right\}
    }{2 C_2^2},
  \end{equation*}
  $\hat{g}^k(\alpha), \hat{G}_i^k(\alpha), \hat{h}^k(\alpha) \ge 0$ are satisfied for all
  $k$ and $\alpha \in (0, \hat{\alpha}]$.
\end{proof}

\subsection{Numerical experiments using a state-of-the-art solver}
\label{section_appendix_SOTA_results}
We additionally report numerical results obtained using a state-of-the-art commercial solver, CPLEX.
All experiments in this subsection were conducted with CPLEX version 22.1.1.
The number of parallel threads was set to one,
and the stopping criteria were chosen to be identical to those described in Section~\ref{section_stopping_criteria}.
All other parameters were kept at their default values.

The computational results are summarized in Table~\ref{table_results_by_cplex}.
Compared with II-arc and II-line,
CPLEX achieves significantly fewer iterations and less computation times across all tested instances.

\begin{table}[htb]
  \centering
  \caption{CPLEX results: number of IPM iterations and computation time in seconds}
  \label{table_results_by_cplex}
  \begin{tabular}{l|rr}
    \hline
    problem  &
    \multicolumn{2}{c}{CPLEX} \\
             & Itr. & Time    \\
    \hline \hline
    CRE-A    & 26   & 1.25    \\
    CRE-B    & 40   & 5.24    \\
    CRE-C    & 29   & 0.95    \\
    CRE-D    & 38   & 4.05    \\
    DFL001   & 25   & 4.99    \\
    KEN-07   & 15   & 1.01    \\
    KEN-11   & 18   & 6.78    \\
    KEN-13   & 25   & 15.43   \\
    KEN-18   & 28   & 53.43   \\
    OSA-07   & 20   & 3.74    \\
    OSA-14   & 30   & 8.31    \\
    OSA-30   & 28   & 15.76   \\
    OSA-60   & 35   & 34.97   \\
    PDS-06   & 27   & 6.04    \\
    PDS-10   & 27   & 11.59   \\
    PDS-20   & 33   & 30.34   \\
    QAP15    & 16   & 7.99    \\
    STOCFOR3 & 24   & 5.86    \\
    \hline
  \end{tabular}
\end{table}

It is worth noting that CPLEX internally computes search directions by solving linear systems based on direct factorization (high-accuracy) solver.
As a result,
the number of iterations is typically smaller than that of methods based on inexact search directions.

Moreover, the reduced computation time of CPLEX is largely attributable to
a variety of sophisticated heuristics implemented in the solver.
These include, for example, advanced strategies for accelerating Cholesky factorizations,
such as ordering techniques, fill-in reduction,
and cache-aware implementations~\cite{lustig1996gigaflops,klotz2013practical}.
In addition, once a sufficiently accurate IPM solution is obtained,
CPLEX may terminate the IPM phase early and switch to the simplex method to compute a high-accuracy optimal solution.
The numbers of iterations reported in Table~\ref{table_results_by_cplex} correspond only to the IPM phase.

While these heuristics lead to excellent practical performance,
they also indicate that CPLEX represents a highly optimized solver
that integrates algorithmic, numerical, and implementation-level techniques developed over several decades.
In contrast, the objective of this paper is not to compete directly with such mature solvers,
but to investigate how arc-search strategies can improve and efficiency of II-IPMs.

Nevertheless, the comparison suggests that incorporating similar heuristic ideas
into the proposed II-arc framework may further improve its practical performance.
Such extensions are beyond the scope of this paper and constitute an important topic for future research.

\bibliographystyle{abbrvnat}
\bibliography{scholar}

@article{mohammadisiahroudi2024efficient,
  title     = {Efficient use of quantum linear system algorithms in inexact infeasible IPMs for linear optimization},
  author    = {Mohammadisiahroudi, Mohammadhossein and Fakhimi, Ramin and Terlaky, Tam{\'a}s},
  journal   = {Journal of Optimization Theory and Applications},
  volume    = {202},
  number    = {1},
  pages     = {146--183},
  year      = {2024},
  publisher = {Springer}
}

@article{wu2023inexact,
  title     = {An inexact feasible quantum interior point method for linearly constrained quadratic optimization},
  author    = {Wu, Zeguan and Mohammadisiahroudi, Mohammadhossein and Augustino, Brandon and Yang, Xiu and Terlaky, Tam{\'a}s},
  journal   = {Entropy},
  volume    = {25},
  number    = {2},
  pages     = {330},
  year      = {2023},
  publisher = {MDPI}
}

@article{mohammadisiahroudi2025inexact,
  title     = {An inexact feasible interior point method for linear optimization with high adaptability to quantum computers},
  author    = {Mohammadisiahroudi, Mohammadhossein and Fakhimi, Ramin and Wu, Zeguan and Terlaky, Tam{\'a}s},
  journal   = {SIAM Journal on Optimization},
  volume    = {35},
  number    = {4},
  pages     = {2203--2233},
  year      = {2025},
  publisher = {SIAM}
}

@article{iida2024infeasible,
  title     = {An infeasible interior-point arc-search method with {N}esterov's restarting strategy for linear programming problems},
  author    = {Iida, Einosuke and Yamashita, Makoto},
  journal   = {Computational Optimization and Applications},
  pages     = {1--34},
  year      = {2024},
  publisher = {Springer}
}

@article{Yang2017,
  author    = {Yaguang Yang},
  doi       = {10.1007/s11075-016-0180-1},
  issn      = {15729265},
  issue     = {4},
  journal   = {Numerical Algorithms},
  month     = {4},
  pages     = {967-996},
  publisher = {Springer New York LLC},
  title     = {Curve{LP}-{A MATLAB} implementation of an infeasible interior-point algorithm for linear programming},
  volume    = {74},
  year      = {2017}
}

@book{yang2020arc,
  title     = {Arc-search techniques for interior-point methods},
  author    = {Yang, Yaguang},
  year      = {2020},
  publisher = {CRC Press},
  address   = {FL}
}

@article{Mehrotra1992,
  author   = {Sanjay Mehrotra},
  doi      = {10.1137/0802028},
  issue    = {4},
  journal  = {SIAM Journal on Optimization},
  pages    = {575-601},
  title    = {On the Implementation of a Primal-Dual Interior Point Method},
  volume   = {2},
  url      = {https://www.researchgate.net/publication/230873223},
  year     = {1992}
}

@book{wright1997primal,
  title     = {Primal-dual interior-point methods},
  author    = {Wright, Stephen J},
  year      = {1997},
  publisher = {SIAM},
  address   = {PA}
}

@article{yang2018arc,
  title     = {An arc-search {$O(nL)$} infeasible-interior-point algorithm for linear programming},
  author    = {Yang, Yaguang and Yamashita, Makoto},
  journal   = {Optimization Letters},
  volume    = {12},
  number    = {4},
  pages     = {781--798},
  year      = {2018},
  publisher = {Springer}
}

@article{Yamashita2021,
  author    = {Makoto Yamashita and Einosuke Iida and Yaguang Yang},
  doi       = {10.1007/s11075-021-01113-w},
  issn      = {15729265},
  journal   = {Numerical Algorithms},
  publisher = {Springer},
  title     = {An infeasible interior-point arc-search algorithm for nonlinear constrained optimization},
  year      = {2021}
}

@article{bellavia1998inexact,
  title     = {Inexact interior-point method},
  author    = {Bellavia, Stefania},
  journal   = {Journal of Optimization Theory and Applications},
  volume    = {96},
  pages     = {109--121},
  year      = {1998},
  publisher = {Springer}
}

@article{kerenidis2020quantum,
  title     = {A quantum interior point method for {LP}s and {SDP}s},
  author    = {Kerenidis, Iordanis and Prakash, Anupam},
  journal   = {ACM Transactions on Quantum Computing},
  volume    = {1},
  number    = {1},
  pages     = {1--32},
  year      = {2020},
  publisher = {ACM New York, NY, USA}
}

@article{al2009convergence,
  title     = {Convergence analysis of the inexact infeasible interior-point method for linear optimization},
  author    = {Al-Jeiroudi, Ghussoun and Gondzio, Jacek},
  journal   = {Journal of Optimization Theory and Applications},
  volume    = {141},
  pages     = {231--247},
  year      = {2009},
  publisher = {Springer}
}

@article{monteiro2003convergence,
  title   = {Convergence analysis of a long-step primal-dual infeasible interior-point {LP} algorithm based on iterative linear solvers},
  author  = {Monteiro, Renato DC and O'Neal, Jerome W},
  journal = {Georgia Institute of Technology},
  year    = {2003}
}

@article{browne1995netlib,
  title   = {The {N}etlib mathematical software repository},
  author  = {Browne, Shirley and Dongarra, Jack and Grosse, Eric and Rowan, Tom},
  journal = {D-Lib Magazine},
  volume  = {1},
  number  = {9},
  year    = {1995}
}

@article{mizuno1999global,
  title   = {Global and polynomial-time convergence of an infeasible-interior-point algorithm using inexact computation.},
  author  = {Mizuno, Shinji and Jarre, Florian},
  journal = {Mathematical Programming},
  volume  = {84},
  number  = {1},
  year    = {1999}
}

@inproceedings{karmarkar1984new,
  title     = {A new polynomial-time algorithm for linear programming},
  author    = {Karmarkar, Narendra},
  booktitle = {Proceedings of the sixteenth annual ACM symposium on Theory of computing},
  pages     = {302--311},
  year      = {1984}
}

@incollection{kojima1989primal,
  title     = {A primal-dual interior point algorithm for linear programming},
  author    = {Kojima, Masakazu and Mizuno, Shinji and Yoshise, Akiko},
  booktitle = {Progress in Mathematical Programming},
  pages     = {29--47},
  year      = {1989},
  publisher = {Springer},
  address   = {New {Y}ork}
}

@article{yang2011polynomial,
  title     = {A polynomial arc-search interior-point algorithm for convex quadratic programming},
  author    = {Yang, Yaguang},
  journal   = {European Journal of Operational Research},
  volume    = {215},
  number    = {1},
  pages     = {25--38},
  year      = {2011},
  publisher = {Elsevier}
}

@article{kojima1993primal,
  title     = {A primal-dual infeasible-interior-point algorithm for linear programming},
  author    = {Kojima, Masakazu and Megiddo, Nimrod and Mizuno, Shinji},
  journal   = {Mathematical programming},
  volume    = {61},
  number    = {1-3},
  pages     = {263--280},
  year      = {1993},
  publisher = {Springer}
}

@article{bellavia2004convergence,
  title     = {Convergence analysis of an inexact infeasible interior point method for semidefinite programming},
  author    = {Bellavia, Stefania and Pieraccini, Sandra},
  journal   = {Computational Optimization and Applications},
  volume    = {29},
  pages     = {289--313},
  year      = {2004},
  publisher = {Springer}
}

@article{dolan2002benchmarking,
  title     = {Benchmarking optimization software with performance profiles},
  author    = {Dolan, Elizabeth D and Mor{\'e}, Jorge J},
  journal   = {Mathematical Programming},
  volume    = {91},
  pages     = {201--213},
  year      = {2002},
  publisher = {Springer}
}

@article{gould2016note,
  title     = {A note on performance profiles for benchmarking software},
  author    = {Gould, Nicholas and Scott, Jennifer},
  journal   = {ACM Transactions on Mathematical Software (TOMS)},
  volume    = {43},
  number    = {2},
  pages     = {1--5},
  year      = {2016},
  publisher = {ACM New York, NY, USA}
}

@misc{orban-benchmarkprofiles-2019,
  author       = {D. Orban and {contributors}},
  title        = {{BenchmarkProfiles.jl}: {A Simple Julia Package to Plot Performance and Data Profiles}},
  month        = {February},
  howpublished = {\url{https://github.com/JuliaSmoothOptimizers/BenchmarkProfiles.jl}},
  year         = {2019},
  doi          = {10.5281/zenodo.4630955}
}

@article{monteiro2004uniform,
  title     = {Uniform boundedness of a preconditioned normal matrix used in interior-point methods},
  author    = {Monteiro, Renato DC and O'Neal, Jerome W and Tsuchiya, Takashi},
  journal   = {SIAM Journal on Optimization},
  volume    = {15},
  number    = {1},
  pages     = {96--100},
  year      = {2004},
  publisher = {SIAM}
}

@article{yang2018two,
  title     = {Two computationally efficient polynomial-iteration infeasible interior-point algorithms for linear programming},
  author    = {Yang, Yaguang},
  journal   = {Numerical Algorithms},
  volume    = {79},
  number    = {3},
  pages     = {957--992},
  year      = {2018},
  publisher = {Springer}
}

@article{yang2023polynomial,
  title     = {A polynomial time infeasible interior-point arc-search algorithm for convex optimization},
  author    = {Yang, Yaguang},
  journal   = {Optimization and Engineering},
  volume    = {24},
  number    = {2},
  pages     = {885--914},
  year      = {2023},
  publisher = {Springer}
}

@article{carolan1990empirical,
  title     = {An empirical evaluation of the {KORBX}{\textregistered} algorithms for military airlift applications},
  author    = {Carolan, William J and Hill, James E and Kennington, Jeffery L and Niemi, Sandra and Wichmann, Stephen J},
  journal   = {Operations Research},
  volume    = {38},
  number    = {2},
  pages     = {240--248},
  year      = {1990},
  publisher = {INFORMS}
}

@article{bartmeyer2021switching,
  title     = {Switching preconditioners using a hybrid approach for linear systems arising from interior point methods for linear programming},
  author    = {Bartmeyer, Petra Maria and Bocanegra, Silvana and Oliveira, Aurelio Ribeiro Leite},
  journal   = {Numerical Algorithms},
  volume    = {86},
  pages     = {397--424},
  year      = {2021},
  publisher = {Springer}
}

@article{lustig1992implementing,
  title     = {On implementing Mehrotra’s predictor--corrector interior-point method for linear programming},
  author    = {Lustig, Irvin J and Marsten, Roy E and Shanno, David F},
  journal   = {SIAM Journal on Optimization},
  volume    = {2},
  number    = {3},
  pages     = {435--449},
  year      = {1992},
  publisher = {SIAM}
}

@article{vitor2022projected,
  title     = {Projected orthogonal vectors in two-dimensional search interior point algorithms for linear programming},
  author    = {Vitor, Fabio and Easton, Todd},
  journal   = {Computational Optimization and Applications},
  volume    = {83},
  number    = {1},
  pages     = {211--246},
  year      = {2022},
  publisher = {Springer}
}

@article{monteiro1990polynomial,
  title     = {A polynomial-time primal-dual affine scaling algorithm for linear and convex quadratic programming and its power series extension},
  author    = {Monteiro, Renato DC and Adler, Ilan and Resende, Mauricio GC},
  journal   = {Mathematics of Operations Research},
  volume    = {15},
  number    = {2},
  pages     = {191--214},
  year      = {1990},
  publisher = {INFORMS}
}

@article{gondzio1996multiple,
  title     = {Multiple centrality corrections in a primal-dual method for linear programming},
  author    = {Gondzio, Jacek},
  journal   = {Computational {O}ptimization and {A}pplications},
  volume    = {6},
  number    = {2},
  pages     = {137--156},
  year      = {1996},
  publisher = {Springer}
}

@article{espaas2022interior,
  title     = {An Interior Point Framework Employing Higher-Order Derivatives of Central Path-like Trajectories: Application to Convex Quadratic Programming},
  author    = {Espaas, Thomas A and Vassiliadis, Vassilios S},
  journal   = {Computers \& Chemical Engineering},
  volume    = {158},
  pages     = {107638},
  year      = {2022},
  publisher = {Elsevier}
}

@article{bocanegra2007using,
  title     = {Using a hybrid preconditioner for solving large-scale linear systems arising from interior point methods},
  author    = {Bocanegra, Silvana and Campos, Frederico F and Oliveira, Aurelio RL},
  journal   = {Computational Optimization and Applications},
  volume    = {36},
  pages     = {149--164},
  year      = {2007},
  publisher = {Springer}
}

@article{oliveira2005new,
  title     = {A new class of preconditioners for large-scale linear systems from interior point methods for linear programming},
  author    = {Oliveira, Aurelio RL and Sorensen, Danny C},
  journal   = {Linear Algebra and Its Applications},
  volume    = {394},
  pages     = {1--24},
  year      = {2005},
  publisher = {Elsevier}
}

@article{gonzaga1990polynomial,
  title     = {Polynomial affine algorithms for linear programming},
  author    = {Gonzaga, Clovis C},
  journal   = {Mathematical Programming},
  volume    = {49},
  pages     = {7--21},
  year      = {1990},
  publisher = {Springer}
}

@article{mizuno1993adaptive,
  title     = {On adaptive-step primal-dual interior-point algorithms for linear programming},
  author    = {Mizuno, Shinji and Todd, Michael J and Ye, Yinyu},
  journal   = {Mathematics of Operations Research},
  volume    = {18},
  number    = {4},
  pages     = {964--981},
  year      = {1993},
  publisher = {INFORMS}
}

@article{renegar1988polynomial,
  title     = {A polynomial-time algorithm, based on {N}ewton's method, for linear programming},
  author    = {Renegar, James},
  journal   = {Mathematical Programming},
  volume    = {40},
  number    = {1},
  pages     = {59--93},
  year      = {1988},
  publisher = {Springer}
}

@article{freund1999convergence,
  title     = {Convergence of a class of inexact interior-point algorithms for linear programs},
  author    = {Freund, Roland W and Jarre, Florian and Mizuno, Shinji},
  journal   = {Mathematics of Operations Research},
  volume    = {24},
  number    = {1},
  pages     = {50--71},
  year      = {1999},
  publisher = {INFORMS}
}

@article{korzak2000convergence,
  title     = {Convergence analysis of inexact infeasible-interior-point algorithms for solving linear programming problems},
  author    = {Korzak, Janos},
  journal   = {SIAM Journal on Optimization},
  volume    = {11},
  number    = {1},
  pages     = {133--148},
  year      = {2000},
  publisher = {SIAM}
}

@article{yang2017arc,
  title     = {An arc-search infeasible-interior-point method for symmetric optimization in a wide neighborhood of the central path},
  author    = {Yang, Ximei and Liu, Hongwei and Zhang, Yinkui},
  journal   = {Optimization Letters},
  volume    = {11},
  number    = {1},
  pages     = {135--152},
  year      = {2017},
  publisher = {Springer}
}

@article{zhang2019primal,
  title     = {A primal-dual interior-point algorithm with arc-search for semidefinite programming},
  author    = {Zhang, Mingwang and Yuan, Beibei and Zhou, Yiyuan and Luo, Xiaoyu and Huang, Zhengwei},
  journal   = {Optimization Letters},
  volume    = {13},
  number    = {5},
  pages     = {1157--1175},
  year      = {2019},
  publisher = {Springer}
}

@article{bergamaschi2021new,
  title     = {A new preconditioning approach for an interior point-proximal method of multipliers for linear and convex quadratic programming},
  author    = {Bergamaschi, Luca and Gondzio, Jacek and Mart{\'{\i}}nez, {\'A}ngeles and Pearson, John W and Pougkakiotis, Spyridon},
  journal   = {Numerical Linear Algebra with Applications},
  volume    = {28},
  number    = {4},
  pages     = {e2361},
  year      = {2021},
  publisher = {Wiley Online Library}
}

@article{scott2018schur,
  title     = {A Schur complement approach to preconditioning sparse linear least-squares problems with some dense rows},
  author    = {Scott, Jennifer and T{\r{u}}ma, Miroslav},
  journal   = {Numerical Algorithms},
  volume    = {79},
  pages     = {1147--1168},
  year      = {2018},
  publisher = {Springer}
}

@article{zhang1994convergence,
  title     = {On the convergence of a class of infeasible interior-point methods for the horizontal linear complementarity problem},
  author    = {Zhang, Yin},
  journal   = {SIAM Journal on Optimization},
  volume    = {4},
  number    = {1},
  pages     = {208--227},
  year      = {1994},
  publisher = {SIAM}
}

@article{mohammadisiahroudi2018improvements,
  title     = {Improvements to quantum interior point method for linear optimization},
  author    = {Mohammadisiahroudi, Mohammadhossein and Wu, Zeguan and Augustino, Brandon and Carr, Arielle and Terlaky, Tam{\'a}s},
  journal   = {ACM Transactions on Quantum Computing},
  year      = {2018},
  publisher = {ACM New York, NY}
}

@inproceedings{dikin1967iterative,
  title     = {Iterative solution of problems of linear and quadratic programming},
  author    = {Dikin, Iliya Iosiphovich},
  booktitle = {Soviet Math. Dokl.},
  volume    = {8},
  pages     = {674--675},
  year      = {1967}
}

@book{nesterov1994interior,
  title     = {Interior-point polynomial algorithms in convex programming},
  author    = {Nesterov, Yurii and Nemirovskii, Arkadii},
  year      = {1994},
  publisher = {SIAM},
  address   = {PA}
}

@article{klotz2013practical,
  title     = {Practical guidelines for solving difficult linear programs},
  author    = {Klotz, Ed and Newman, Alexandra M},
  journal   = {Surveys in Operations Research and Management Science},
  volume    = {18},
  number    = {1-2},
  pages     = {1--17},
  year      = {2013},
  publisher = {Elsevier}
}

@article{lustig1996gigaflops,
  title     = {Gigaflops in linear programming},
  author    = {Lustig, Irvin J and Rothberg, Edward},
  journal   = {Operations Research Letters},
  volume    = {18},
  number    = {4},
  pages     = {157--165},
  year      = {1996},
  publisher = {Elsevier}
}

\end{document}